\pdfoutput=1
\documentclass[a4paper,12pt]{amsart}

\usepackage{amsmath,amssymb,amsxtra,amscd,mathrsfs}
\usepackage[margin=3cm]{geometry}
\usepackage{color}  
\usepackage{tikz}
\usepackage{hyperref} 
\usepackage[all]{xy}

\numberwithin{equation}{section}
\allowdisplaybreaks[3]

\newtheorem{theorem}{Theorem}[section]
\newtheorem{lemma}[theorem]{Lemma} 
\newtheorem{proposition}[theorem]{Proposition} 
\newtheorem{corollary}[theorem]{Corollary} 
\theoremstyle{definition}
 
\newtheorem{notation}[theorem]{Notation} 
\newtheorem{remark}[theorem]{Remark} 
\newtheorem{example}[theorem]{Example}

\newcommand{\C}{\mathbb{C}} 

\newcommand{\Z}{\mathbb{Z}} 
\newcommand{\Q}{\mathbb{Q}} 
\newcommand{\R}{\mathbb{R}}
\newcommand{\PP}{\mathbb{P}} 
\newcommand{\bS}{\mathbb{S}}

\newcommand{\hJ}{\widehat{J}} 
\newcommand{\htau}{\hat{\tau}}
\newcommand{\hP}{\widehat{P}}
\newcommand{\hlambda}{\hat{\lambda}}
\newcommand{\hGamma}{\widehat{\Gamma}} 

\newcommand{\bt}{\mathbf{t}} 
\newcommand{\bc}{\mathbf{c}} 
\newcommand{\bs}{\mathbf{s}} 
\newcommand{\btau}{\boldsymbol{\tau}}

\newcommand{\sfH}{\mathsf{H}}

\newcommand{\te}{\tilde{e}} 
\newcommand{\tDelta}{\widetilde{\Delta}} 
\newcommand{\tI}{\widetilde{I}} 
\newcommand{\tJ}{\widetilde{J}} 
\newcommand{\tF}{\widetilde{F}} 

\newcommand{\tsigma}{\tilde{\sigma}} 
\newcommand{\tupsilon}{\tilde{\upsilon}} 

\newcommand{\frakm}{\mathfrak{m}} 

\newcommand{\cO}{\mathcal{O}}
\newcommand{\cS}{\mathcal{S}} 
\newcommand{\cF}{\mathcal{F}} 
\newcommand{\cH}{\mathcal{H}} 
 
\newcommand{\cL}{\mathcal{L}} 
\newcommand{\tcL}{\widetilde{\cL}} 

\newcommand{\cQ}{\mathcal{Q}}

\newcommand{\scrF}{\mathscr{F}}

\newcommand{\iu}{\sqrt{-1}} 

\newcommand{\rank}{\operatorname{rank}} 
\newcommand{\Ker}{\operatorname{Ker}} 
\newcommand{\ev}{\operatorname{ev}} 
\newcommand{\pt}{\operatorname{pt}} 
 
\newcommand{\Spf}{\operatorname{Spf}} 
\newcommand{\End}{\operatorname{End}} 
\newcommand{\Res}{\operatorname{Res}} 
\newcommand{\ch}{\operatorname{ch}} 
\newcommand{\Map}{\operatorname{Map}} 
\newcommand{\Eff}{\operatorname{Eff}} 

\newcommand{\QDM}{\operatorname{QDM}} 
\newcommand{\FT}{\operatorname{FT}} 
\newcommand{\id}{\operatorname{id}} 
\newcommand{\Fl}{\operatorname{Fl}}

\newcommand{\Hom}{\operatorname{Hom}}

\def\corr#1{\left\langle#1 \right\rangle} 
\def\parfrac#1#2{\frac{\partial #1}{\partial #2}} 

\begin{document} 

\title{Quantum cohomology of projective bundles} 
\author{Hiroshi Iritani}
\email{iritani@math.kyoto-u.ac.jp} 
\address{Department of Mathematics, Graduate School of Science, 
Kyoto University, Kitashirakawa-Oiwake-cho, Sakyo-ku, 
Kyoto, 606-8502, Japan}

\author{Yuki Koto} 
\email{ykoto@gate.sinica.edu.tw} 
\address{Institute of Mathematics, Academia Sinica, 
Astronomy-Mathematics Building, No.1, Sec.4, 
Roosevelt Road, Taipei 10617, Taiwan.} 


\begin{abstract} 
We construct an $I$-function of the projective bundle $\PP(V)$ associated with a not necessarily split vector bundle $V\to B$ as a Fourier transform of the $S^1$-equivariant $J$-function of the total space of $V$ and show that it lies on the Givental Lagrangian cone of $\PP(V)$. 
Using this result, we show that the quantum cohomology $D$-module of $\PP(V)$ splits into a direct sum of the quantum cohomology $D$-modules of the base space $B$. This has applications to the semisimplicity of big quantum cohomology. 
\end{abstract} 

\maketitle 

\section{Introduction} 


For a rank-$r$ vector bundle $V \to B$ over a smooth projective variety $B$, we can construct the projective bundle $\pi\colon \PP(V) \to B$ consisting of lines in the fibres of $V\to B$. The ordinary cohomology group $H^*(\PP(V))$ of $\PP(V)$ is generated by the relative hyperplane class $p := c_1(\cO_{\PP(V)}(1))$ as an $H^*(B)$-algebra as follows: 
\[
H^*(\PP(V)) \cong  H^*(B)[p]/(p^r + c_1(V)p^{r-1} + \cdots + c_r(V) ).  
\]
This isomorphism is a consequence of the Leray-Hirsch theorem that says that $H^*(\PP(V))$ is a free $H^*(B)$-module generated by $1,p,\dots,p^{r-1}$. 

Many researchers have pursued a quantum cohomology counterpart to the Leray-Hirsch theorem for projective bundles or even broader fibre bundles, see e.g.~\cite{Qin-Ruan:projective, Costa--Miro-Roig:projective, Ancona-Maggesi:bundles, Elezi:projective_bundle, Leung-Li:functorial, Brown:toric_fibration, Strangeway, Lee-Lin-Wang:flopII, Lee-Lin-Qu-Wang:flopIII, Gonzalez-Woodward:tmmp, Fan:quantum_splitting, Fan-Lee:projective_bundle, Fan:projective_bundle}. 
If we count only `vertical' curves contracted by the projection $\PP(V) \to B$, we can easily observe that the corresponding `small vertical' quantum cohomology of $\PP(V)$ is given by 
\[
H^*(B)[p,q]/(p^r + c_1(V) p^{r-1} + \cdots + c_r(V) - q) 
\]
where $q$ is the Novikov variable associated with vertical lines. 
For $q\neq 0$, this ring decomposes into a direct sum of $r$ copies of $H^*(B)$. This decomposition must persist under the full quantum deformation, parametrized by the full Novikov variables of $\PP(V)$ and the bulk-deformation parameter $\htau \in H^*(\PP(V))$. The question then arises as to whether each summand is isomorphic to the quantum cohomology of $B$.  
This paper aims to affirm this correspondence, i.e.~$QH^*(\PP(V))$ decomposes into a direct sum of $r$ copies of $QH^*(B)$. Furthermore, we show that the decomposition can be lifted to quantum cohomology $D$-modules.  

\subsection{Mirror theorem for non-split projective bundles} 
When the vector bundle $V \to B$ is \emph{split}, i.e.~a direct sum of line bundles, the Elezi-Brown mirror theorem \cite{Elezi:projective_bundle, Brown:toric_fibration} enables us to determine the quantum cohomology of $\PP(V)$. 
This theorem provides an explicit hypergeometric series, called the \emph{$I$-function}, that lies in the Givental cone of $\PP(V)$. 
Recall that the Givental cone of a smooth projective variety $X$ is a Lagrangian submanifold in the infinite dimensional symplectic vector space $\cH_X = H^*(X)(\!(z^{-1})\!)[\![Q]\!]$. It encodes all genus-zero Gromov-Witten invariants of $X$. 
The Elezi-Brown $I$-function for $\PP(V)$ reconstructs the Givenal cone and hence determines all the genus-zero Gromov-Witten invariants of $\PP(V)$. 

Our first main result is a generalization of the Elezi-Brown mirror theorem to the non-split case. It is stated in terms of the $S^1$-equivariant Gromov-Witten invariants of the vector bundle $V \to B$. 
\begin{theorem}[Theorem \ref{thm:mirrorthm}] 
\label{thm:mirrorthm_introd} 
Let $B$ be a smooth projective variety and let $V \to B$ be a vector bundle of rank $r\ge 2$. We assume that the dual bundle $V^\vee$ is generated by global sections. Consider the $S^1$-action on $V$ scaling fibres and let $J_V^\lambda(\tau)$ denote the $S^1$-equivariant $J$-function of the total space of $V$, where  $\lambda$ is the $S^1$-equivariant parameter. Define the $H^*(\PP(V))$-valued function $I_{\PP(V)}(\tau,t)$ by  
\[
I_{\PP(V)}(\tau,t) = \sum_{k=0}^\infty 
\frac{e^{pt/z} q^k e^{kt} }{\prod_{c=1}^k \prod_{\substack{\text{\rm $\delta$: Chern roots} \\ \text{\rm of $V$}}} (p+ \delta + cz)} J_V^{p+kz}(\tau).    
\]
Then $z I_{\PP(V)}(\tau,t)|_{z\to -z}$ lies in the Givental cone of $\PP(V)$. Here $q$ denotes the Novikov variable corresponding to the class of a line in a fibre and $p$ is the first Chern class of the relative $\cO(1)$ over $\PP(V)$. 
\end{theorem}

The function $I_{\PP(V)}$ in this theorem coincides with the Elezi-Brown $I$-function when $V$ is a direct sum of line bundles (see Example \ref{exa:split_case}). 
It arises as the Fourier transform of the equivariant $J$-function $J_V^\lambda$ with respect to the equivariant parameter $\lambda$. We shall elaborate on this point in \S\ref{subsec:introd_Fourier} below. 


\begin{remark} 
By tensoring $V$ with a sufficiently negative line bundle, we can always assume that $V^\vee$ is generated by global sections, without changing $\PP(V)$. This assumption ensures that $J_V^\lambda$ does not contain negative powers of $\lambda$ and that the substitution $\lambda = p +kz$ in the theorem is well-defined (see \S\ref{subsec:vector_bundle}). 
The conclusion of Theorem \ref{thm:mirrorthm_introd} holds true even if we replace the $J$-function $J_V^\lambda$ with any slices of the Givental cone of $V$ depending polynomially on $\lambda$; see Theorem \ref{thm:mirrorthm} for the details. 
\end{remark} 

\begin{remark} The $J$-function $J_V^\lambda$ is defined over the Novikov ring of $V$. In Theorem \ref{thm:mirrorthm_introd}, we use the splitting of $\pi_* \colon H_2(\PP(V),\Z) \twoheadrightarrow H_2(B,\Z)=H_2(V,\Z)$ given by the class $p=c_1(\cO(1))$ to identify Novikov variables for $V$ with those for $\PP(V)$, see Notation \ref{nota:splitting}. 
\end{remark} 


\begin{remark} 
The quantum Riemann-Roch theorem of Coates-Givental \cite{Coates-Givental} determines the $S^1$-equivariant $J$-function $J_V^\lambda$ from the genus-zero Gromov-Witten invariants of the base $B$ and the Chern classes $c_i(V)$. Therefore Theorem \ref{thm:mirrorthm_introd} determines all genus-zero Gromov-Witten invariants of $\PP(V)$ from those of the base $B$. This recovers the genus-zero part of the results of Fan-Lee \cite{Fan-Lee:projective_bundle} and Fan \cite{Fan:projective_bundle}, who showed that all genus\footnote{The method in this paper is not applicable to higher-genus reconstruction, as it relies on the quantum Lefschetz principle \cite{KKP:functoriality} for convex bundles.} 
Gromov-Witten invariants of $\PP(V)$ can be reconstructed from those of the base $B$ and the Chern classes $c_i(V)$. 
On the other hand, in contrast to the split case, it is not easy to find a closed formula for $J_V^\lambda$ in general. We give a non-split example, the cotangent bundle of $\PP^n$, in Example \ref{exa:cotangent_Pn}. 
\end{remark} 

\begin{remark}
\label{rem:calculation_GW_of_PV}
Several techniques exist for computing the Gromov-Witten invariants of non-split projective bundles. Strangeway \cite{Strangeway} provided a reconstruction algorithm for certain projective bundles over $\PP^n$. The quantum splitting principle of Lee-Lin-Qu-Wang \cite{Lee-Lin-Qu-Wang:flopIII} proposes a reduction to the split case by blowing up the base. Fan \cite[Theorem 0.1, Question 4]{Fan:quantum_splitting} explored a relationship between the $(V,e_\lambda^{-1})$-twisted Gromov-Witten invariants of $B$ and the $(\cO(-1),e_\lambda^{-1})$-twisted (or untwisted) invariants of $\mathbb{P}(V)$, under the opposite assumption that $V$ is globally generated. Recently, Hinault-Yu-Zhang-Zhang \cite{HYZZ:framing} (see also earlier announcements \cite{Kontsevich:Miami2020, Kontsevich:Simons2021} by Katzarkov-Kontsevich-Pantev-Yu) proposed a reconstruction method based on the decomposition theorem (Theorem \ref{thm:decomp_QDM_introd} below). The initial condition required for their reconstruction and the explicit algorithm based on Birkhoff factorization are provided in \S \ref{subsec:initial_condition}. 
\end{remark}

\begin{remark} 
In the subsequent paper \cite{Koto:non-split_toric}, the second-named author generalized Theorem \ref{thm:mirrorthm_introd} to toric bundles associated with non-split vector bundles, combining the techniques in this paper and virtual localization for an additional torus action. 
\end{remark} 

\subsection{Decomposition of quantum $D$-modules} 

An important property of quantum cohomology is that it forms a $D$-module over the parameter space. Roughly speaking, the \emph{quantum $D$-module} $\QDM(X)$ of a smooth projective variety $X$ is the module $H^*(X)[z][\![Q,\tau]\!]$ equipped with the flat covariant derivatives $\nabla_{\tau^i}$, $\nabla_{z\partial_z}$, $\nabla_{\xi Q\partial_Q}$ with respect to the parameters $\tau$, $z$, $Q$ (see \S\ref{subsec:qconn_fundsol}). Here $\tau \in H^*(X)$ represents the bulk parameter and $Q$ denotes the Novikov variable. These covariant derivatives, called the quantum connection, are defined in terms of the quantum multiplication. The quantum $D$-module is further equipped with a $z$-sesquilinear pairing $P$, induced from the Poincar\'e pairing and compatible with the quantum connection. The $I$- or $J$-functions appearing in the previous section are cohomology-valued solutions of the quantum $D$-module. 

The Mirror Theorem \ref{thm:mirrorthm_introd} establishes that the quantum $D$-module $\QDM(\PP(V))$ is the Fourier transform of the $S^1$-equivariant quantum $D$-module $\QDM_{S^1}(V)$ (see Theorem \ref{thm:Fourier_QDM_introd} below). By combining this result with the quantum Riemann-Roch (QRR) theorem of Coates-Givental \cite{Coates-Givental}, which connects the Gromov-Witten invariants of $V$ and $B$, we prove that $\QDM(\PP(V))$ decomposes into a direct sum of $r$ copies of $\QDM(B)$ after an appropriate localization. 

To compare the quantum $D$-modules of $\PP(V)$ and $B$, we embed the Novikov ring of $B$ into a localization of the Novikov ring of $\PP(V)$. Let $(q,Q)$ denote the Novikov variable of $\PP(V)$, where $q$ corresponds to the class of a vertical line and $Q$ is the Novikov variable of the base $B$, identified via the splitting in Notation \ref{nota:splitting}. Introducing an additional copy $Q_B$ of the Novikov variable for $B$, we define the following embedding\footnote{See Remark \ref{rem:splitting_intrinsic_meaning} for the geometric interpretation of this embedding.} of Novikov rings: 
\begin{equation}
\label{eq:embedding_Novikov}
\C[\![Q_B]\!] \hookrightarrow \C(\!(q^{-1/r'})\!)[\![Q]\!] \qquad Q_B^d \mapsto q^{-c_1(V)\cdot d/r} Q^d  
\end{equation} 
where $r'$ is $r$ or $2r$ depending on the parity of $r$ (see \eqref{eq:r'}). 
Let $\QDM(\PP(V))_{\rm loc}$ denote the base change of $\QDM(\PP(V))$ to the localized base $\C(\!(q^{-1/r'})\!)[\![Q]\!]$, and let $\QDM(B)_{\rm ext, loc}$ denote the base change of $\QDM(B)$ via the embedding \eqref{eq:embedding_Novikov}.   



\begin{figure}[t]  
\centering 
\begin{tikzpicture}[x=1.2pt, y=1.2pt] 
\draw[very thick] (0,50) .. controls (80,60) and (120,60)  .. (200,50);
\draw[->, very thick] (70,56.5) -- (71,56.5); 
\draw[->,very thick] (30,30) .. controls (30,70) and (25,80) .. (20,100); 
\draw[->, very thick] (170,30) .. controls (170,65) and (175,80) .. (180,100); 
\draw[->,very thick] (140,56) -- (139,56); 

\draw (20,108) node {$Q$};
\draw (180,108) node {$Q_B=q^{-c_1(V)/r} Q$}; 
\draw (70,45) node {$q$}; 
\draw (140,45) node {$q^{-1/r'}$};

\filldraw[opacity=0.1, blue, rotate=3] (55,53) ellipse [x radius =80, y radius =30];
\draw (65,75) node {$\QDM(\PP(V))$}; 

\filldraw[opacity=0.4, rotate=-5] (165,70) ellipse [x radius =10,y radius=20];
\draw (213,68) node {$\displaystyle \QDM(B)_{\rm ext}^{\boxplus r}$};
\end{tikzpicture} 
\caption{Decomposition pictured over the global K\"ahler moduli space: $\QDM(\PP(V))$ splits into a direct sum of $\QDM(B)$ near the infinity divisor $(q^{-1/r'}=0)$. }
\label{fig:decomposition} 
\end{figure}
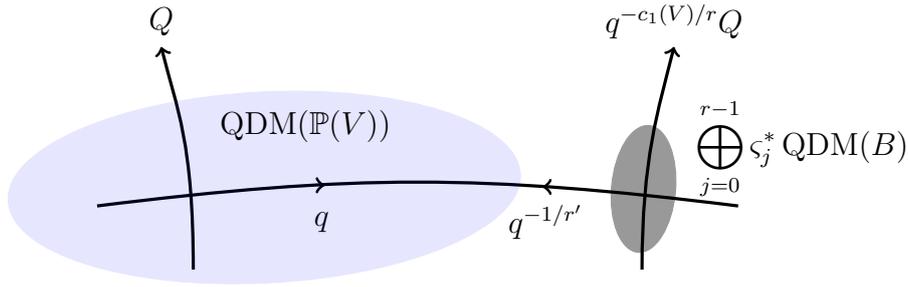 

\begin{theorem}[Theorem \ref{thm:decomp_QDM}] 
\label{thm:decomp_QDM_introd} 
Let $V\to B$ be as in Theorem $\ref{thm:mirrorthm_introd}$. We have a formal invertible map $\varsigma=\bigoplus_{j=0}^{r-1} \varsigma_j \colon H^*(\PP(V)) \to \bigoplus_{j=0}^{r-1} H^*(B)$ over $\C(\!(q^{-1/r})\!)[\![Q]\!]$ and a decomposition of the quantum $D$-modules over $\C[z](\!(q^{-1/r'})\!)[\![Q]\!]$: 
\[
\Phi\colon 
\QDM(\PP(V))_{\rm loc} \cong \bigoplus_{i=0}^{r-1} \varsigma_i^* \QDM(B)_{\rm ext, loc}.  
\]
\end{theorem} 

The pairs $(q,Q)$ and $(q^{-1/r'}, Q_B)$ of Novikov variables, related by $Q_B = q^{-c_1(V)/r} Q$, may be viewed as two coordinate charts on the global K\"ahler moduli space illustrated in Figure \ref{fig:decomposition}. The isomorphism $\Phi$ glues the quantum $D$-modules $\QDM(\PP(V))$ and $\QDM(B)_{\rm ext}^{\boxplus r} = \QDM(B)[\![q^{-1/r'}]\!]^{\boxplus r}$ over the intersection of their domains of definition --- the localized base $\Spf\C(\!(q^{-1/r'})\!)[\![Q]\!]$ --- to form a global $D$-module.

\begin{corollary}[see \S\ref{subsec:semisimplicity} for the proof] 
\label{cor:decomp_qcoh} The derivative of the map $\varsigma$ induces an isomorphism of the quantum cohomology rings $(H^*(\PP(V)), \star_{\htau}) \cong \bigoplus_{i=0}^{r-1} (H^*(B), \star_{\varsigma_i(\htau)})$ over the localized base $\C(\!(q^{-1/r})\!)[\![Q]\!]$. In particular, the quantum cohomology of $\PP(V)$ is generically semisimple if and only if the same holds for the base $B$.  
\end{corollary} 

We can use Corollary \ref{cor:decomp_qcoh} to deduce the generic semisimplicity of quantum cohomology for certain Grassmannians $G(k,n)$, $IG(k,2n)$ (see Proposition \ref{prop:semisimplicity}). The semisimplicity result seems new for $IG(k,2n)$ with $k\ge 3$ (see \cite{Perrin:semisimple,GMS:IG26,CMPS:IG} for  $k=2$). 

\begin{remark} 
We do not discuss the convergence of quantum cohomology. In this paper, ``generic semisimplicity'' means formal semisimplicity in the sense of \cite{Coates-Iritani:convergence}. 
\end{remark} 

The proofs of Theorems \ref{thm:mirrorthm_introd} and \ref{thm:decomp_QDM_introd} both rely on the QRR theorem \cite{Coates-Givental}. We embed $V$ to the trivial bundle $B\times \C^N$. This induces an embedding $\PP(V) \subset B\times \PP^{N-1}$, which exhibits $\PP(V)$ as the zero-locus of a section of a convex vector bundle $\cQ(1)$ over $B\times \PP^{N-1}$. Theorem \ref{thm:mirrorthm_introd} follows by applying the QRR theorem to the bundle $\cQ(1)$. 
The Decomposition Theorem \ref{thm:decomp_QDM_introd} is established by studying a (formal) stationary  phase approximation of a certain Fourier integral \eqref{eq:Fourier_J} of $J_V^\lambda$. The QRR theorem for the bundle $V\to B$ implies that this integral takes values in the Givental cone of $B$, thereby defining a morphism $\QDM(\PP(V)) \to \QDM(B)$ of $D$-modules.


\begin{remark} 
In the split case, a similar decomposition of the quantum $D$-module (more generally, for toric bundles) follows from Brown's theorem \cite{Brown:toric_fibration}: this has been discussed by Coates-Givental-Tseng \cite{CGT:Virasoro} and Koto \cite{Koto:convergence}. These previous works rely on the equivariant localization with respect to the $(\C^\times)^{r-1}$-action (which exists only in the split case) and our argument is substantively different from them.   
\end{remark} 

\begin{remark}[see \cite{Iritani:discrepant, Kontsevich:Miami2020, Iritani:ICM}]  
\label{rem:Orlov}
Theorem \ref{thm:decomp_QDM_introd} provides a decomposition of quantum $D$-modules at the formal level. Given that quantum cohomology is expected to have convergent structure constants, we anticipate that the mirror map $\varsigma_j$ and the isomorphism $\Phi$ in Theorem \ref{thm:decomp_QDM_introd} can be analytified with respect to the Novikov variables $q,Q$ and the bulk parameter $\htau \in H^*(\PP(V))$. However, we should not expect $\Phi$ to be analytic in $z$; rather, its matrix entries should belong to $\cO^{\rm an}[\![z]\!]$, where $\cO^{\rm an}$ denotes the sheaf of analytic functions in $q, Q, \htau$. This formal (in $z$) decomposition should lift to an analytic decomposition over an angular sector in the $z$-plane. Furthermore, we expect this analytic decomposition to relate --- at specific parameter values --- to the semi-orthogonal decomposition \cite{Orlov:projective} of the derived category of coherent sheaves:
\[
D^b(\PP(V)) =\langle D^b(B)_0,\dots,D^b(B)_{r-1} \rangle, \qquad D^b(B)_i = D^b(B)\otimes \cO(i)
\]
via the $\hGamma$-integral structure. 
\end{remark}

\subsection{Equivariant quantum cohomology and Fourier transformation}
\label{subsec:introd_Fourier} 
We explain the idea behind our construction of the $I$-function. Following the conjecture of Teleman \cite{Teleman:gauge_mirror} on quantum cohomology of symplectic reductions, we view the quantum $D$-module of $\PP(V)$ as a Fourier transform of the equivariant quantum $D$-module of $V$. We refer to \cite{Iritani:monoidal, Iritani-Sanda:reduction} for a more general conjecture along this line of thought. 

Let $(M,\omega)$ be a sympletic manifold equipped with a Hamiltonian $S^1$-action, and let $\mu \colon M\to \R$ be a moment map. We denote by $M/\!\!/_{t} S^1$ the symplectic reduction $\mu^{-1}(t)/S^1$ at the level $t\in \R$. The equivariant volume of $M$ can be written as a Fourier transform of the volume of $M/\!\!/_t S^1$: 
\[
\int_M e^{\omega - \lambda \mu} = \int_\R  e^{-\lambda t} dt \int_{M/\!\!/_t S^1} e^{\omega_{\rm red}} 
\]
where $\lambda$ is the $S^1$-equivariant parameter and $\omega_{\rm red}$ is the reduced symplectic form on $M/\!\!/_t S^1$. 
This formula has played a fundamental role in the works of Duistermaat-Heckman  \cite{Duistermaat-Heckman} and Jeffrey-Kirwan \cite{Jeffrey-Kirwan:localization}. We want to ask if a similar relation continues to hold in equivariant \emph{quantum} cohomology. Since the $J$-function is a quantum analogue of $e^{\omega}$, we could perhaps hope that the equivariant $J$-function of $M$ is a Fourier transform of the $J$-function of the reduction $M/\!\!/_t S^1$. 

We can make this hope more precise by using shift (or Seidel) operators in equivariant quantum cohomology. The $S^1$-equivariant quantum $D$-module $\QDM_{S^1}(M)$ of $M$ admits a shift operator $\bS \colon \QDM_{S^1}(M) \to \QDM_{S^1}(M)$ of the equivariant parameter $\lambda$ \cite{Okounkov-Pandharipande:Hilbert}. This lifts the Seidel operator \cite{Seidel:pi1} on quantum cohomology and satisfies $\bS( f(\lambda) \alpha) = f(\lambda-z) \bS(\alpha)$ for any $f(\lambda) \in \C[\lambda]$ and $\alpha \in \QDM_{S^1}(M)$.  We have the following commutation relation: 
\[
[\lambda, \bS] = z \bS.  
\]
Now we change the point of view and regard $\QDM_{S^1}(M)$ as a $\C[\bS]$-module rather than a $\C[\lambda]$-module; the equivariant parameter $\lambda$ then behaves like the differential operator $z \bS \parfrac{}{\bS}$. In other words, we consider the Fourier transform of the module $\QDM_{S^1}(M)$.  Teleman's conjecture suggests that the Fourier transform of $\QDM_{S^1}(M)$ corresponds to $\QDM(M/\!\!/_t S^1)$ in such a way that the $\lambda$-action gives the quantum connection in the direction of $\kappa(\lambda)$ and that the $\bS$-action gives the Novikov variable action dual to $\kappa(\lambda)$, where $\kappa \colon H^*_{S^1}(M) \to H^*(M/\!\!/_t S^1)$ denotes the Kirwan map. 

According to this picture, \emph{solutions} of $\QDM_{S^1}(M)$ and $\QDM(M/\!\!/_t S^1)$ should be related by a Fourier transformation. Let us restrict ourselves to the case where $M$ is the total space of $V\to B$. The equivariant $J$-function $J_V^\lambda$ gives a solution of $\QDM_{S^1}(V)$ which takes values in the (polynomial) Givental space $\cH_{V,\rm pol}= H^*_{S^1}(V)[z,z^{-1}][\![Q]\!]$. The shift action on the Givental space $\cH_{V,\rm pol}$ is given by the operator $\cS \colon \cH_{V,\rm pol} \to \cH_{V,\rm pol}$ 
\begin{equation} 
\label{eq:cS_introd}
\cS (f(\lambda)) = e_{\lambda}(V)  f(\lambda -z)  
\end{equation} 
where $e_{\lambda}(V)=\sum_{i=0}^r \lambda^i c_{r-i}(V)$ is the $S^1$-equivariant Euler class of $V$ (see \cite[Definition 3.13]{Iritani:shift}). We consider the following Fourier transformation from $\cH_{V,\rm pol}$ to $\cH_{\PP(V),\rm pol}$:  
\begin{equation} 
\label{eq:Fourier_introd}
J(\lambda) \longmapsto \hJ(q) = \sum_{k\in \Z}  \kappa( \cS^{-k} J) q^k 
\end{equation} 
where $\kappa\colon H^*_{S^1}(V) \to H^*(\PP(V))$ is the Kirwan map, mapping $\lambda$ to $p=c_1(\cO(1))$, and $\cH_{\PP(V),\rm pol} = H^*(\PP(V)) [z,z^{-1}][\![q,Q]\!]$ is the polynomial Givental space for $\PP(V)$. It is easy to see that this transformation intertwines $\lambda$ with $zq\partial_q+p$, and $\cS$ with $q$.    
\begin{align}
\label{eq:Fourier_property} 
(\lambda J)\sphat  =  \left(z q \partial_q + p \right) \hJ, \qquad 
(\cS J)\sphat  = q \hJ. 
\end{align} 
Hence the Fourier transform of $J_V^\lambda$ should give rise to a solution of $\QDM(\PP(V))$. It is calculated as 
\begin{align}
\label{eq:Fourier_J-function} 
(J_V^\lambda)\sphat & = \sum_{k\in \Z} \kappa \left( \frac{\prod_{c=-\infty}^0 e_{\lambda+cz}(V)}{\prod_{c=-\infty}^k e_{\lambda+cz}(V)} J_V^{\lambda+kz} \right) q^k  
=\sum_{k\ge 0} \frac{q^k}{\prod_{c=1}^k e_{p+cz}(V)} J_V^{p+kz}.  
\end{align} 
The summand with $k<0$ vanishes because of the relation $\kappa(e_{\lambda}(V)) = p^r + c_1(V) p^{r-1} + \cdots + c_r(V)= 0$ in $H^*(\PP(V))$. This gives the $I$-function $I_{\PP(V)}$ in Theorem \ref{thm:mirrorthm_introd} restricted to $t=0$. 

We can rephrase the mirror theorem (Theorems \ref{thm:mirrorthm_introd}, \ref{thm:mirrorthm}) as the Fourier duality between the quantum $D$-modules of $V$ and $\PP(V)$. In this paper, we treat $\QDM_{S^1}(V)$ as a $D$-module over the infinite-dimensional base $H^*_{S^1}(V)$ (see \S\ref{subsec:vector_bundle} and \S\ref{subsubsec:QDM_V} for further details). 

\begin{theorem}[Theorem \ref{thm:Fourier_QDM}, Proposition \ref{prop:Fourier_pairing}] 
\label{thm:Fourier_QDM_introd} 
Let $V\to B$ be as in Theorem $\ref{thm:mirrorthm_introd}$. There exist a mirror map $\htau \colon H^*_{S^1}(V) \to H^*(\PP(V))$, whose leading term is the Kirwan map, and an isomorphism $\FT \colon \QDM_{S^1}(V) \cong \htau^*\QDM(\PP(V))$ such that $\FT$ intertwines 
\begin{itemize} 
\item the shift operator $\bS$ with the Novikov variable $q$ and; 
\item the equivariant parameter $\lambda$ with the quantum connection $z(\htau^*\nabla)_{q\partial_q}$. 
\end{itemize} 
Moreover, the natural pairings on $\QDM_{S^1}(V)$ and $\QDM(\PP(V))$ can be related by a Fourier transformation. 
\end{theorem} 


\subsection{Plan of the paper} 
In \S\ref{sec:prelim}, we review Gromov-Witten invariants, Givental formalism and shift operators. We prove a technical Lemma \ref{lem:miniversal} on miniveral elements of the Givental cone, which plays a crucial role in the proof of the mirror theorem. In \S \ref{sec:proof_mirrorthm}, we prove our mirror theorem and give examples. In \S\ref{sec:Fourier_QDM}, we reformulate the mirror theorem as a Fourier duality between quantum $D$-modules of $V$ and $\PP(V)$.  In \S\ref{sec:decomp}, we prove a decomposition of the quantum $D$-module of $\PP(V)$. 

\subsection*{Acknowledgements} 
H.I.~thanks Fumihiko Sanda, Yuuji Tanaka and Constantin Teleman for very useful discussions on shift operators and equivariant cohomology. 
We thank anonymous referees for many valuable comments. 
H.I.~is supported by JSPS grant 16H06335, 20K03582, 21H04994 and 23H01073. Y.K.~is supported by JSPS grant 21J23023. 

\section{Preliminaries} 
\label{sec:prelim} 
We introduce notation on Gromov-Witten invariants and quantum cohomology. Then we discuss quantum connection, the Givental Lagrangian cone, quantum Riemann-Roch (QRR) theorem and equivariant shift operators. 
\subsection{Gromov-Witten invariants and quantum cohomology}
A reference on Gromov-Witten invariants is made to \cite{Manin:Frobenius_book}. Let $X$ be a smooth projective variety over $\C$. For $n\ge 0$ and $d\in H_2(X,\Z)$, let $X_{0,n,d}$ denote the moduli stack of genus-zero, $n$-points, degree $d$ stable maps to $X$. This is a proper Deligne-Mumford stack carrying a virtual fundamental class $[X_{0,n,d}]_{\rm vir}\in H_{2D}(X_{0,n,d},\Q)$ of degree $D=\dim_\C X + c_1(TX)\cdot d+ n-3$. Let $\ev_i \colon X_{0,n,d} \to X$ be the evaluation map at the $i$th marked point. The (genus-zero, descendant) \emph{Gromov-Witten invariants} are the integrals of the form: 
\begin{equation*}
\corr{\alpha_1\psi^{k_1},\dots,\alpha_n\psi^{k_n}}_{0,n,d}^X 
= \int_{[X_{0,n,d}]_{\rm vir}} 
\prod_{i=1}^n \ev_i^*(\alpha_i) \psi_i^{k_i} 
\end{equation*} 
where $\alpha_1,\dots,\alpha_n \in H^*(X)$, $k_1,\dots,k_n$ are nonnegative integers, and $\psi_i$ is the first Chern class of the universal cotangent line bundle at the $i$th marked point. 

Let $\Eff(X) \subset H_2(X,\Z)$ denote the semigroup of effective curve classes. For a module $K$, we write $K[\![Q]\!]$ for the space of formal power series $\sum_{d\in \Eff(X)}k_d Q^d$ with $k_d \in K$. The formal variable $Q$ is called the \emph{Novikov variable}. The module $K[\![Q]\!]$ has a natural ring structure when $K$ is a ring. The quantum cohomology will be defined over the ring $\C[\![Q]\!]$, which is called the \emph{Novikov ring}. 

Let $(\alpha,\beta)_X = \int_X \alpha \cup \beta$ denote the Poincar\'e pairing on $H^*(X)$. Let $\{\phi_i\}_{i=0}^s$ be a homogeneous basis of $H^*(X)$. We denote by $\tau =\sum_{i=0}^s \tau^i \phi_i$ a general point in $H^*(X)$; then $\tau^0,\dots, \tau^s$ form a $\C$-linear coordinate system on $H^*(X)$. The quantum product $\star_\tau$ at a parameter $\tau \in H^*(X)$ is defined by 
\begin{equation} 
\label{eq:quantum_product}
(\alpha\star_\tau \beta,\gamma)_X = \sum_{n\ge 0} \sum_{d\in \Eff(X)} 
\corr{\alpha,\beta,\gamma,\tau,\cdots,\tau}_{0,n+3,d}^X \frac{Q^d}{n!}  
\end{equation} 
for $\alpha,\beta,\gamma \in H^*(X)$. We call $\tau$ the \emph{bulk(-deformation) parameter}. The product $\star_\tau$ introduces an associative and  supercommutative ring structure on $H^*(X)[\![Q,\tau]\!]=H^*(X)[\![Q]\!][\![\tau^0,\dots,\tau^s]\!]$ such that $1\in H^0(X)$ is the unity. This is called \emph{quantum cohomology}. The variables $\tau^i$ associated with odd classes $\phi_i$ have odd parity, and therefore anti-commute with each other by supercommutativity; more precisely we have $\tau^i \tau^j = (-1)^{|i| |j|} \tau^j \tau^i$, $\tau^i\phi_j = (-1)^{|i||j|} \phi_j \tau^i$ with $|i|= \deg \phi_i$. The Novikov variable $Q$ is even and commutes with other elements. We also follow the Koszul sign convention when expanding the correlator $\corr{\alpha,\beta,\gamma,\tau,\dots,\tau}_{0,n+3,d}^X$ in power series of $\tau^i$. The parameter space $H^*(X)$ should be viewed as a formal supermanifold over $\C[\![Q]\!]$; we refer to  \cite{Manin:Frobenius_book} for the details. 

\subsection{Quantum connection and fundamental solution}
\label{subsec:qconn_fundsol}
The \emph{quantum connection} is a meromorphic flat connection $\nabla$ on the vector bundle $H^*(X) \times (H^*(X)\times \C) \to H^*(X)\times \C$ defined by the formulae
\begin{align}
\label{eq:qconn} 
\begin{split} 
\nabla_{\tau^i} & = \partial_{\tau^i} + z^{-1}(\phi_i\star_\tau) \\
\nabla_{z \partial_z} & = z\partial_z  - z^{-1} (E_X\star_\tau) + \mu_X 
\end{split} 
\end{align} 
where $(\tau,z)$ denotes a point on the base $H^*(X)\times \C$, $E_X$ is the Euler vector field (a section of this vector bundle) 
\begin{equation} 
\label{eq:Euler}
E_X = c_1(TX) + \sum_{i=0}^s \left(1- \frac{\deg \phi_i}{2}\right) \tau^i \phi_i 
\end{equation} 
and $\mu_X\in \End(H^*(X))$ is the grading operator 
\begin{equation} 
\label{eq:mu_grading} 
\mu_X (\phi_i) = \left(\frac{\deg \phi_i}{2} - \frac{\dim_\C X}{2}\right) \phi_i.
\end{equation} 
More precisely, the base space $H^*(X)\times \C$ should be considered as a formal supermanifold over $\C[\![Q]\!]$ whose ring of regular functions is $\C[z][\![Q,\tau]\!]$. The flat connection can be also extended in the direction of Novikov variables $Q$: we define for $\xi \in H^2(X)$ 
\begin{equation} 
\label{eq:qconn_Novikov}
\nabla_{\xi Q\partial_Q} = \xi Q \partial_Q + z^{-1} (\xi\star_\tau) 
\end{equation} 
where $\xi Q \partial_Q$ denotes the derivation of $\C[\![Q]\!]$ given by $(\xi Q\partial_Q) Q^d = (\xi \cdot d) Q^d$. 
In algebraic terms, the quantum connection defines pairwise supercommuting operators 
\[
\nabla_{\tau^i}, \ \nabla_{z\partial_z}, \ \nabla_{\xi Q\partial_Q}  \colon H^*(X)[z][\![Q,\tau]\!] \to z^{-1} H^*(X)[z][\![Q,\tau]\!].  
\]
The quantum connection admits the \emph{fundamental solution} $M_X(\tau)$ 
defined as follows (see \cite[\S 1]{Givental:equivariant}, \cite[Proposition 2]{Pandharipande:afterGivental}, \cite[Proposition 2.4]{Iritani:integral}, \cite[Proposition 3.1]{CCIT:MS} for the properties below): 
\begin{equation} 
\label{eq:fundsol}
(M_X(\tau) \phi_i,\phi_j)_X  = (\phi_i,\phi_j)_X + 
\sum_{\substack{d\in \Eff(X), n\ge 0 \\ (d,n) \neq (0,0)}}
\corr{\phi_i,\tau,\dots,\tau,\frac{\phi_j}{z-\psi}}_{0,n+2,d}^X \frac{Q^d}{n!} 
\end{equation} 
where $\phi_j/(z-\psi)$ in the correlator should be expanded as $\sum_{b=0}^\infty \phi_j \psi^b/z^{b+1}$. 
The solution $M_X(\tau)$ belongs to $\End(H^*(X))\otimes \C[z^{-1}][\![Q,\tau]\!]$ and satisfies the following differential equations: 
\begin{align}
\label{eq:diffeq_M}  
\begin{split}
M_X(\tau) \circ \nabla_{\tau^i}&= \partial_{\tau^i} \circ M_X(\tau), \\ 
M_X(\tau) \circ \nabla_{\xi Q\partial_Q} & = \left(\xi Q\partial_Q + z^{-1}\xi \right)\circ  M_X(\tau), \\
M_X(\tau) \circ \nabla_{z \partial_z} & =\left(z \partial_z - z^{-1} c_1(TX) + \mu_X \right) \circ M_X(\tau)
\end{split}  
\end{align} 
where $\xi$ and $c_1(X)$ on the right-hand side act by the cup product. Let $P_X$ denote the $z$-sesquilinear pairing on $H^*(X) \otimes \C[z,z^{-1}][\![Q,\tau]\!]$ induced by the Poincar\'e pairing 
\begin{equation} 
\label{eq:P} 
P_X(f,g) = (f(-z),g(z))_X. 
\end{equation} 
Then the pairing $P_X$ is flat for $\nabla$ and $M_X$ is isometric with respect to $P_X$. 
\begin{align}
\label{eq:pairing_properties} 
\begin{split} 
d P_X(f,g) & = P_X(\nabla f, g) + P_X(f, \nabla g) \\ 
P_X(f,g) & = P_X(M_X(\tau)f, M_X(\tau)g) 
\end{split}
\end{align} 
where $f,g \in H^*(X)\otimes \C[z][\![Q,\tau]\!]$. 
The \emph{$J$-function} of $X$ is defined to be the image of $1$ under $M_X(\tau)$: 
\begin{align} 
\label{eq:J-function}
\begin{split}  
J_X(\tau) & = M_X(\tau)1 = 1 + \sum_{(d,n)\neq (0,0)} \sum_{j=0}^s  \corr{1,\tau,\dots,\tau,\frac{\phi^j}{z-\psi}}_{0,n+2,d}^X \phi_j \frac{Q^d}{n!} \\ 
& = 1 + \frac{\tau}{z} + \sum_{(d,n)\neq (0,0), (0,1)} \sum_{k=0}^\infty \sum_{j=0}^s \frac{\phi^j}{z^{k+2}}\corr{\tau,\dots,\tau,\phi_j \psi^{k}}_{0,n+1,d} 
\frac{Q^d}{n!} 
\end{split} 
\end{align} 
where $\{\phi^j\}$ is the dual basis of $\{\phi_i\}$ such that $(\phi_i,\phi^j)_X = \delta_i^j$. We used the String Equation \cite[\S 1.2]{Pandharipande:afterGivental} in the second line. The other columns $M_X(\tau)\phi_i$ of $M_X(\tau)$ are given by the derivatives $z\partial_{\tau^i}J_X(\tau)$ of the $J$-function due to the differential equation \eqref{eq:diffeq_M}. 

\subsection{The vector bundle case} 
\label{subsec:vector_bundle} 
In this paper, we also consider the quantum cohomology of a non-compact space, the total space of a vector bundle $V \to B$ of rank $r$ over a smooth projective base $B$. We introduce the fibrewise scalar $\C^\times$-action on $V$. This induces a $\C^\times$-action on the moduli space $V_{0,n,d}$, and the virtual class and the $\psi$-classes have canonical equivariant lifts. The $\C^\times$-fixed locus of $V_{0,n,d}$ is identified with $B_{0,n,d}$ and hence is proper. Thus we can define the \emph{equivariant Gromov-Witten invariants}  
\[
\corr{\alpha_1 \psi^{k_1},\dots,\alpha_n \psi^{k_n}}_{0,n,d}^V = \int_{[V_{0,n,d}]_{\rm vir}}^{\C^\times} \prod_{i=1}^n \ev_i^*(\alpha_i) \psi_i^k 
\]
with $\alpha_1,\dots,\alpha_n \in H^*_{\C^\times}(V) = H^*_{S^1}(V) = H^*(B) \otimes \C[\lambda]$ by the virtual localization formula \cite{Graber-Pandharipande}, where $\lambda$ is the $\C^\times$-equivariant (or $S^1$-equivariant) parameter\footnote{In this paper, all $S^1$-actions are induced from $\C^\times$-actions, so we use the terms `$S^1$-equivariant' and `$\C^\times$-equivariant' interchangeably.}. The equivariant Gromov-Witten invariants take values in $\C[\lambda,\lambda^{-1}]$. In \S\ref{subsec:QRR} below, we describe them as the Gromov-Witten invariants of $B$ twisted by $V$ and the inverse equivariant Euler class. 

Let $\{\phi_i\}_{i=0}^s$ be a homogeneous basis of $H^*(B)$. This induces a $\C[\lambda]$-basis $\{\phi_i\}_{i=0}^s$ of $H^*_{S^1}(V)=H^*(B)\otimes \C[\lambda]$ and  $\C[\lambda]$-valued coordinates $\tau^0,\dots,\tau^s$ on $H_{S^1}^*(V)$; $(\tau^0,\dots,\tau^s)$ parametrizes a point $\tau = \sum_{i=0}^s \tau^i \phi_i \in H^*_{S^1}(V)$.  We define the quantum product $\star_\tau$ on $H^*_{S^1}(V)$ by the same formula \eqref{eq:quantum_product}, replacing the Poincar\'e pairing and the Gromov-Witten invariants with their equivariant counterparts. Note that the equivariant Poincar\'e pairing of $V$ is given by the localization formula: 
\begin{equation} 
\label{eq:pairing_V}
(\alpha,\beta)_V := \int_B \alpha \cup \beta \cup e_\lambda(V)^{-1} \qquad 
\text{for $\alpha,\beta \in H^*_{S^1}(V) = H^*(B) \otimes \C[\lambda]$}
\end{equation}
where $e_\lambda(V) = \sum_{i=0}^r \lambda^i c_{r-i}(V)$ is the equivariant Euler class. In general, the quantum product $\alpha\star_\tau \beta$ with $\alpha,\beta \in H^*_{S^1}(V)$ lies only in the localization $H^*_{S^1}(V)_{\rm loc} := H^*_{S^1}(V) \otimes_{\C[\lambda]}  \C[\lambda,\lambda^{-1}]$. As noted in \cite[\S 1.4]{Bryan-Graber}, we can define $\alpha\star_\tau \beta$ without inverting $\lambda$ if $V$ is semi-projective, i.e.~projective over an affine variety. In this case, the evaluation map $\ev_i$ is proper, and therefore the quantum product can be defined as the push-forward along the evaluation map. We show the following. 
\begin{lemma}
\label{lem:semiprojective} 
Suppose that $V^\vee$ is generated by global sections. Then $V$ is semi-projective. In particular, the quantum product $\star_\tau$ is defined on $H^*_{S^1}(V)[\![Q,\tau]\!]$ without localization. 
\end{lemma} 
\begin{proof} 
By the assumption, we have a surjection $\cO^{\oplus N} \to V^\vee$ for some $N>0$. Taking the dual, we find that $V$ is a subbundle of the trivial bundle $B\times \C^N$. 
The closed embedding $V \hookrightarrow B \times \C^N$ composed with the projection $B \times \C^N \to \C^N$ gives a projective morphism $V \to \C^N$. Thus $V$ is projective over an affine variety. 
\end{proof} 

Hereafter we shall assume that $V^\vee$ is generated by global sections. The quantum connection for $V$ is defined similarly to \eqref{eq:qconn} and \eqref{eq:qconn_Novikov} except that the $\tau$-parameter space $H^*_{S^1}(V)$ is now infinite-dimensional over $\C$. Consider the $\C$-basis $\{\phi_i \lambda^k\}_{0\le i\le s, k\ge 0}$ of $H^*_{S^1}(V)$ and introduce the \emph{$\C$-linear} coordinates $\{\tau^{i,k}\}$ dual to $\{\phi_i \lambda^k\}$ so that $\tau^i = \sum_{k=0}^\infty \tau^{i,k} \lambda^k$. The quantum connection is given by 
\begin{align*} 
\nabla_{\tau^{i,k}} & = \partial_{\tau^{i,k}} + z^{-1} (\phi_i \lambda^k \star_\tau) \\
\nabla_{\xi Q\partial_Q} & = \xi Q\partial_Q + z^{-1} (\xi \star_\tau) \\ 
\nabla_{z\partial_z} & = z\partial_z - z^{-1}(E_V\star_\tau) + \mu_V 
\end{align*} 
where $\xi \in H^2(V) = H^2(B)$ and $E_V$, $\mu_V$ are given by
\begin{align*} 
E_V &= c_1^{S^1}(TV) + \sum_{i=0}^s \sum_{k\ge 0} \left(1 - \frac{\deg \phi_i}{2}-k\right) \tau^{i,k} \phi_i \lambda^k \\
\mu_V (\phi_i \lambda^k)  &= \left(\frac{\deg \phi_i}{2} + k - \frac{\dim B + r}{2} \right) \phi_i \lambda^k 
\end{align*} 
with $c_1^{S^1}(TV) = c_1(TB) + c_1(V) + r\lambda$ the equivariant first Chern class. Note that $\mu_V$ contains the derivation $\lambda \partial_\lambda$. These operators define maps 
\[
\nabla_{\tau^{i,k}}, \ \nabla_{\xi Q\partial_Q}, \ \nabla_{z\partial_z} \colon H^*_{S^1}(V)[z][\![Q,\btau]\!] \to z^{-1} H^*_{S^1}(V)[z][\![Q,\btau]\!] 
\]
where $K[\![\btau]\!]$ denotes the space of formal power series in $\btau=\{\tau^{i,k}\}_{0\le i\le s, k\ge 0}$ with coefficients in $K$ for a module $K$ (we use the boldface symbol $\btau$ to distinguish infinitely many parameters $\{\tau^{i,k}\}$ from finitely many ones  $\{\tau^i\}$). 

We can define the fundamental solution $M_V(\tau)$ for $V$ similarly to \eqref{eq:fundsol} by replacing the Poincar\'e pairing and the Gromov-Witten invariants with their equivariant counterparts. It satisfies the properties \eqref{eq:diffeq_M}, \eqref{eq:pairing_properties} with $X$ and $c_1(TX)$ there replaced with $V$ and $c_1^{S^1}(TV)$  respectively. Since $V$ is semi-projective and the evaluation maps are proper, $M_V(\tau)$ is defined without localization with respect to $\lambda$, i.e.
\begin{equation}
\label{eq:M_V_belongs} 
M_V(\tau) \in \End(H^*(B))\otimes \C[\lambda,z^{-1}][\![Q,\tau]\!]. 
\end{equation} 
We note that the coefficient of $M_V(\tau)$ in front of a fixed monomial $Q^d \tau^I$ is a polynomial in $z^{-1}$ (instead of a formal power series in it). This follows from the virtual localization formula: the $\psi$-classes $\psi_i$ are nilpotent on the fixed locus $B_{0,n,d}$. These properties ensure that the substitution $\lambda \to p + kz$ in $J_V(\tau) = M_V(\tau)1$ (in Theorem \ref{thm:mirrorthm_introd}) is well-defined. 
\begin{remark}
Some of the quantities here (such as $\star_\tau$, $M_V(\tau)$, $J_V(\tau)$) are defined over $\C[\lambda][\![\tau]\!]$ while others (such as $E_V$, $\nabla_{z\partial_z}$) are defined over $\C[\lambda][\![\btau]\!]$.  
Note that we have the natural inclusion $\C[\lambda][\![\tau]\!] \subset \C[\lambda][\![\btau]\!]$ sending $\tau^i$ to $\sum_{k\ge 0} \tau^{i,k}\lambda^k$. This inclusion is, however, not preserved by $\lambda\partial_\lambda$ and hence we need to work with $\C[\lambda][\![\btau]\!]$ when considering $\nabla_{z\partial_z}$. 
\end{remark} 

\subsection{Givental Lagrangian cones}
\label{subsec:Givental_cone} 
We briefly review Givental's symplecitc formalism for genus-zero Gromov-Witten theory, as developed in \cite{Givental:symplectic,Coates-Givental}. The \emph{Givental symplectic space} of a smooth projective variety $X$ is the infinite-dimensional $\C[\![Q]\!]$-vector space 
\begin{equation} 
\label{eq:Givental_space} 
\cH_X = H^*(X)(\!(z^{-1})\!)[\![Q]\!] 
\end{equation} 
equipped with the symplectic form 
\[
\Omega(f,g) = -\Res_{z=\infty} (f(-z),g(z))_X dz. 
\]
The Givental space $\cH_X$ has a canonical decomposition $\cH_X = \cH_+\oplus \cH_-$ into maximally isotropic subspaces $\cH_\pm$, where 
\[
\cH_+= H^*(X)[z][\![Q]\!], \quad 
\cH_- = z^{-1} H^*(X)[\![z^{-1}]\!][\![Q]\!].   
\]
The symplectic form $\Omega$ identifies $\cH_-$ with the dual space of $\cH_+$, hence $\cH_X=\cH_+\oplus \cH_-$ with the cotangent bundle $T^* \cH_+$ of $\cH_+$. Let $\cF_X$ denote the genus-zero descendant Gromov-Witten potential:  
\[
\cF_X(-z+\bt(z)) = \sum_{\substack{d\in \Eff(X), n\ge 0\\ (d,n)\neq (0,0), (0,1), (0,2)}} \corr{\bt(\psi),\dots,\bt(\psi)}_{0,n,d}^X \frac{Q^d}{n!}    
\]
where $\bt(z) = \sum_{n\ge 0} t_n z^n$ with $t_n =\sum_{i=0}^s t^i_n\phi_i \in H^*(X)$. The argument $-z + \bt(z)$ represents a point in the formal neighbourhood of $-z$ in $\cH_+$ and $\cF_X$ is regarded as a function on the formal neighbourhood of $-z$ in $\cH_+$. The \emph{Givental Lagrangian cone} $\cL_X^{\rm orig} \subset \cH_X$ is a submanifold-germ defined as the graph of $d \cF_X$ in $T^*\cH_+ \cong \cH_X$. It consists of 
\begin{multline}
\label{eq:point_on_the_cone} 
-z + \bt(z) + \sum_{k\ge 0} \sum_{i=0}^s \frac{\phi^i}{(-z)^{k+1}} \parfrac{\cF_X}{t_k^i} \\
= -z + \bt(z) + \sum_{\substack{d\in \Eff(X),n\ge 0 \\ (d,n) \neq (0,0), (0,1)}} \sum_{k\ge 0} \sum_{i=0}^s   \frac{\phi^i}{(-z)^{k+1}} \corr{\bt(\psi),\dots,\bt(\psi),\phi_i \psi^k}_{0,n+1,d}^X 
\frac{Q^d}{n!}
\end{multline} 
with $\bt(z) \in \cH_+$. More rigorously, for a set $x=(x_1,x_2,\dots)$ of formal parameters, we say that an element $f \in \cH_X[\![x]\!]$ is a \emph{$\C[\![Q,x]\!]$-valued point} of $\cL_X^{\rm orig}$ if $f$ is of the form \eqref{eq:point_on_the_cone} for some $\bt(z) \in \cH_+[\![x]\!]$ with $\bt(z)|_{Q=x=0} =0$. We refer to \cite[Appendix B]{CCIT:computing} for the definition of $\cL_X$ as an infinite-dimensional formal scheme over $\C[\![Q]\!]$. The nilpotence\footnote{This may fail in equivariant theory, and this is a reason why we allow infinitely many negative powers of $z$ in general. However we do not deal with such cases in this paper.} of the $\psi$-class implies that $\cL_X^{\rm orig}$ is contained in the following polynomial form 
\begin{equation} 
\label{eq:H_X_pol}
\cH_{X,\rm pol} = H^*(X)\otimes \C[z,z^{-1}][\![Q]\!],  
\end{equation} 
i.e.~a $\C[\![Q,x]\!]$-valued point on $\cL_X^{\rm orig}$ is contained in $\cH_{X,\rm pol}[\![x]\!]$. 

\begin{example} 
\label{exa:J} 
After multiplying by $z$ and flipping the sign of $z$, the $J$-function $J_X(\tau)$ in \eqref{eq:J-function} lies on the $\cL_X^{\rm orig}$: 
\[
z J_X(\tau)|_{z\to -z} \in \cL_X^{\rm orig}.  
\]
It defines a $\C[\![Q,\tau]\!]$-valued point on $\cL_X^{\rm orig}$. 
\end{example} 
\begin{notation} 
\label{nota:flip} 
The $J$-function and the Givental cone have opposite sign conventions for $z$. In the following discussion, to ease the notation, we shall instead use the Givental cone $\cL_X := \cL_X^{\rm orig}|_{z\to -z}$ with the sign of $z$ flipped, so that $z J_X(\tau) \in \cL_X$. 
\end{notation} 

As the name suggests, $\cL_X$ is a cone, i.e.~the dilation vector field  on $\cH_X$ (the vector field assigning $f$ to any point $f\in \cH_X$) is tangent to $\cL_X$. Moreover it satisfies the following ``overruled'' property \cite[Theorem 1]{Givental:symplectic}: 
\begin{quote}
\it the tangent space $T$ of $\cL_X$ at any point of $\cL_X$ is tangent to $\cL_X$ precisely along $zT\subset T$. 
\end{quote} 
Since $\cL_X$ is a cone, the tangent space $T$ at $f\in \cL_X$ (regarded as an affine subspace of $\cH_X$) contains both the origin and $f$. The overruled property says that (1) $T \cap \cL_X = z T$ and (2) the tangent space of $\cL_X$ at any point on $zT$ equals $T$. It follows that the Givental cone is a union of semi-infinite subspaces $zT$ over all tangent spaces $T$. Moreover, the tangent spaces to $\cL_X$ form a finite-dimensional family: every tangent space equals the tangent space $T_\tau$ at $z J_X(\tau)$ for some $\tau$, and $T_\tau$ is a free $\C[z][\![Q]\!]$-module generated by the derivatives $z \partial_{\tau^i}J_X(\tau) = M_X(\tau)\phi_i$. Hence we have 
\begin{equation} 
\label{eq:cone_as_a_union} 
\cL_X = \bigcup_{\tau} z M_X(\tau)\cH_+. 
\end{equation} 
A precise meaning of this equality is as follows: any $\C[\![Q,x]\!]$-valued point on $\cL_X$ can be written as $z M_X(\tau) f$ for some $\tau\in H^*(X)\otimes \C[\![Q,x]\!]$ and $f \in \cH_+[\![x]\!]$ such that $\tau|_{Q=x=0} = 0$ and $f|_{Q=x=0} = 1$. 


\subsubsection{The vector bundle case} 
\label{subsubsec:equivariant_cone}
We explain the Givental cone associated with equivariant Gromov-Witten invariants of the vector bundle $V\to B$. We again assume that $V^\vee$ is generated by global sections so that $V$ is semi-projective (see Lemma \ref{lem:semiprojective}). The Givental symplectic space $\cH_V$ for $V$ is defined similarly to \eqref{eq:Givental_space} by replacing $H^*(X)$ with $H^*_{S^1}(V)$. 
\[
\cH_V = H^*_{S^1}(V)(\!(z^{-1})\!)[\![Q]\!] = H^*(B)\otimes \C[\lambda](\!(z^{-1})\!)[\![Q]\!].  
\]
The genus-zero descendant Gromov-Witten potential $\cF_V$ for $V$, defined in terms of the equivariant Gromov-Witten invariants of $V$, takes values in $\C[\lambda,\lambda^{-1}][\![Q]\!]$. The properness of the evaluation maps implies that the points \eqref{eq:point_on_the_cone} on the Givental cone with $X=V$ do not contain negative powers of $\lambda$ whenever the input $\bt(z)$ lies in $\cH_+= H^*_{S^1}(V)[z][\![Q]\!]$. Hence we get the Givental cone $\cL_V$ for $V$ as a subset of $\cH_V$. To be precise, with notation as in \S\ref{subsec:vector_bundle}, a $\C[\lambda][\![Q,x]\!]$-valued point of $\cL_V$ is a point of the form 
\[
z + \bt(z) +  \sum_{\substack{d\in \Eff(X),n\ge 0 \\ (d,n) \neq (0,0), (0,1)}} \sum_{k\ge 0} \sum_{i=0}^s   \frac{e_\lambda(V) \phi^i}{z^{k+1}} \corr{\bt(-\psi),\dots,\bt(-\psi),\phi_i \psi^k}_{0,n+1,d}^V 
\frac{Q^d}{n!}
\]
for some $\bt(z) \in \cH_+[\![x]\!] = H^*_{S^1}(V)[z][\![Q,x]\!]$ with $\bt(z)|_{Q=x=0}=0$, 
where $\{\phi^i\}$ is a basis of $H^*(B)$ dual to $\{\phi_i\}$ in the sense that $(\phi_i,\phi^j)_B = \delta_i^j$: hence $(\phi_i, e_\lambda(V) \phi^j)_V=\delta_i^j$.  As discussed at the end of \S\ref{subsec:vector_bundle}, the virtual localization formula furthermore shows that the Givental cone $\cL_V$ is contained in the following polynomial form $\cH_{V,\rm pol} \subset \cH_V$: 
\begin{equation} 
\label{eq:H_V_pol} 
\cH_{V,\rm pol} = H^*_{S^1}(V)\otimes_{\C[\lambda]} \C[\lambda,z,z^{-1}][\![Q]\!]. 
\end{equation} 

\subsection{Twisted Gromov-Witten invariants and quantum Riemann-Roch theorem} 
\label{subsec:QRR} 
We review the twisted Gromov-Witten invariants and the quantum Riemann-Roch (QRR) theorem due to Coates-Givental \cite{Coates-Givental}. The data of the twist are given by a vector bundle and a characteristic class. Let $\bc$ denote the universal multiplicative characteristic class with parameters $\bs =(s_0,s_1,\dots)$:
\[
\bc(\cdot) = \exp\left(\sum_{k=0}^\infty s_k \ch_k(\cdot)\right).  
\]
For a vector bundle $W \to X$, we introduce the $(W,\bc)$-twisted Gromov-Witten invariants of $X$ as 
\[
\corr{\alpha_1\psi^{k_1},\dots,\alpha_n \psi^{k_n}}^{X,(W,\bc)}_{0,n,d} = 
\int_{[X_{0,n,d}]_{\rm vir}} \left(\prod_{i=1}^n \ev_i^*(\alpha_i) \psi_i^{k_i}\right) \cup \bc(W_{0,n,d}) 
\]
where $W_{0,n,d}=\R\pi_* f^* W\in K^0(X_{0,n,d})$ is the complex of vector bundles constructed by the following universal family of stable maps:
\[
\xymatrix{C_{0,n,d}
\ar[d]_{\pi} \ar[r]^{f} & X \\
X_{0,n,d} & 
}
\]
We can similarly define the twisted versions of quantum cohomology, quantum connection, fundamental solution, $J$-function, Givental symplectic space and Givental cone, by using the twisted Gromov-Witten invariants and the twisted Poincar\'e pairing 
\begin{equation} 
\label{eq:tw_Poincare} 
(\alpha,\beta)_{X,(W,\bc)} = \int_X \alpha \cup \beta \cup \bc(W).  
\end{equation} 

In this paper, the twisted Gromov-Witten invariants appear mainly  in the following two cases. 
\begin{itemize} 
\item[(a)] $W$ is convex, i.e.~we have $H^1(C,f^*W)=0$ for any genus-zero stable map $C\to X$ and $\bc$ is the equivariant Euler class $e_\lambda$. This corresponds to the choice of parameters 
\[
s_k= 
\begin{cases} 
\log \lambda & k=0; \\ 
(-1)^{k-1} (k-1)! \lambda^{-k} & k\ge 1.
\end{cases}
\]

\item[(b)] $W^\vee$ is generated by global sections and $\bc$ is the \emph{inverse} equivariant Euler class $e_\lambda^{-1}$. This corresponds to the negative of the parameters $s_k$ given above. 
\end{itemize} 
In the case (a), the $(W,e_\lambda)$-twisted Gromov-Witten invariants are closely related to the Gromov-Witten invariants of the zero-locus $Z \subset X$ of a regular section of $W$. The convexity assumption implies that $W_{0,n,d}$ is represented by a vector bundle. Hence we can take the non-equivariant limit $\lambda \to 0$ of the $(W,e_\lambda)$-twisted invariants. The functoriality of virtual classes  \cite{KKP:functoriality} implies 
\[
\lim_{\lambda \to 0} \corr{\alpha_1 \psi^{k_1},\dots,\alpha_n \psi^{k_n}}_{0,n,d}^{X,(W,e_\lambda)} 
=\sum_{i_* d' = d} \corr{i^*(\alpha_1) \psi^{k_1},\dots, i^*(\alpha_n) \psi^{k_n}}_{0,n,d'}^Z 
\]
where $i\colon Z \to X$ is the inclusion map. Moreover, the fundamental solution $M_{X,(W,e_\lambda)}$ in the $(W,e_\lambda)$-twisted theory satisfies 
(see \cite[\S 2.1]{Pandharipande:afterGivental}, \cite[Proposition 2.4]{Iritani:periods})  
\[
\lim_{\lambda \to 0} i^* M_{X,(W,e_\lambda)}(\tau) \phi = M_Z(i^*\tau) i^*\phi \Bigr|_{H_2(Z,\Z) \to H_2(X,\Z)}
\]
for $\tau,\phi\in H^*(X)$, where the subscript $H_2(Z,\Z) \to H_2(X,\Z)$ means to replace the Novikov variable $Q^d$ with $Q^{i_*d}$ for $d\in H_2(Z,\Z)$. This means that points on the $(W,e_\lambda)$-twisted Givental cone yield, after applying $i^*$ and taking the non-equivariant limit, points on the Givental cone of $Z$. 

The case (b) corresponds to the equivariant Gromov-Witten invariants of the total space of $W$. The virtual localization formula \cite{Graber-Pandharipande} implies that the equivariant Gromov-Witten invariants of the vector bundle $V\to B$ equal the $(V,e_\lambda^{-1})$-twisted Gromov-Witten invariants of $B$:
\[
\corr{\alpha_1\psi^{k_1},\dots,\alpha_n \psi^{k_n}}^V_{0,n,d} = 
\corr{\alpha_1\psi^{k_1},\dots,\alpha_n \psi^{k_n}}^{B,(V,e_\lambda^{-1})}_{0,n,d}. 
\]

We now state the (genus-zero) QRR theorem of Coates-Givental \cite{Coates-Givental}. Let $\cH_{X,(W,\bc)} = H^*(X)(\!(z^{-1})\!)[\![Q]\!][\![\bs]\!]$ denote the Givental space for the $(W,\bc)$-twisted theory: it is equipped with the symplectic form induced by the twisted Poincar\'e pairing \eqref{eq:tw_Poincare}. Define the symplectic operator\footnote{In view of our convention from Notation \ref{nota:flip}, we flipped the sign of $z$ from the original definition.} $\Delta_{(W,\bc)}\colon \cH_X[\![\bs]\!] \to \cH_{X,(W,\bc)}$ as 
\begin{equation} 
\label{eq:symp_operator} 
\Delta_{(W,\bc)} = \exp\left(\sum_{l,m\ge 0} s_{l+m-1} \frac{B_m}{m!} \ch_l(W) (-z)^{m-1}\right)
\end{equation} 
where we set $s_{-1}=0$ and $\{B_m\}_{m=0}^\infty$ are the Bernoulli numbers defined by $\sum_{m=0}^\infty \frac{B_m}{m!} x^m = x/(e^x-1)$; we have $B_0=1,B_1=-\frac{1}{2},B_2=\frac{1}{6}$, etc., and $B_{2k+1} =0$ for all $k\ge 1$. Let $\cL_{X,(W,\bc)}\subset \cH_{X,(W,\bc)}$ denote the Givental cone for the $(W,\bc)$-twisted theory. 
\begin{theorem}[{\cite[Corollary 4]{Coates-Givental}}] 
\label{thm:QRR} 
We have $\cL_{X,(W,\bc)}=\Delta_{(W,\bc)} \cL_X$. 
\end{theorem}

\begin{remark} 
\label{rem:modified_Euler_twist} 
In the case of $e_\lambda$ or $e_\lambda^{-1}$-twist, we have $s_0=\pm \log \lambda$ whereas $s_k$'s with positive $k$ are proportional to $\lambda^{-k}$. The existence of $\log \lambda$ makes it difficult to treat the $e_\lambda^\pm$-twisted Givental cone formally in terms of the $\lambda^{-1}$-adic topology. We deal with the $e_\lambda^\pm$-twist in the following two steps: first we consider the twist by $\te_\lambda^\pm$, where 
\[
\te_\lambda(W) = \sum_{i\ge 0} \lambda^{-i} c_i(W) = \lambda^{-\rank W} e_\lambda(W),
\]  
and get the $\te_\lambda^{\pm}$-twisted Givental cone $\tcL$ defined over $\C[\![Q,\lambda^{-1}]\!]$; then points on the $e_\lambda^\pm$-twisted cone arise from points on $\tcL$ by the change of variables $Q^d \to Q^d \lambda^{\pm c_1(W)\cdot d}$ (this follows easily from the definition of points on the cone, together with $\rank W_{0,n,d} = \rank W + c_1(W) \cdot d$ and the change of the twisted Poincar\'e pairings). 
\end{remark}

\subsection{Moving points on the Givental cone via differential operators} 
\label{subsec:moving} 
We review a well-known technique to construct points on the Givental cone via the action of differential operators, and show that any ``miniversal'' slices on the cone arise from the $J$-function in this way. This plays a key role in the proof of the main theorems. Let $x=(x_1,x_2,\dots)$, $y=(y_1,y_2,\dots)$ be two sets of formal parameters. We consider a differential operator $F=F(x,z\partial_x,z)$ such that 
\[
F\in \C[z][\![x]\!]\langle z \partial_{x_1},z\partial_{x_2},\dots \rangle [\![y,Q]\!], \qquad 
F|_{y=Q=0} = 0.   
\] 
The second condition ensures that $\exp(F/z)$ converges in the $(Q,y)$-adic topology. 
\begin{lemma}[{\cite[\S 8]{Coates-Givental}, \cite[\S 4.2]{CCIT:computing}}] 
\label{lem:flow_on_the_cone}
The action of the differential operator $\exp(F/z)$ preserves $\C[\![Q,x,y]\!]$-valued points on $\cL_X$. 
\end{lemma} 
\begin{proof} This is well-known, at least in the case where $F$ is a constant differential operator (i.e.~do not depend on $x$). We include a brief discussion. The overruled property implies that $z^{-1} F f$ for a $\C[\![Q,x,y]\!]$-valued point $f$ of $\cL_X$ lies in the tangent space of $\cL_X$ at $f$. Thus the assignment $f\mapsto z^{-1}F f$ defines a vector field on the mapping scheme\footnote{In \cite[Appendix B]{CCIT:computing}, the formal scheme $\cL_X$ was given as a functor from topological $\C[\![Q]\!]$-algebras to sets. It is represented by a certain formal power series ring over $\C[\![Q]\!]$. The mapping scheme here is given by the functor $R \mapsto \cL_X(R[\![x,y]\!])$, which is again represented by a formal power series ring over $\C[\![Q]\!]$.} $\Map(\Spf(\C[\![Q,x,y]\!]),\cL_X)$ and $\exp(F/z)f$ is the time-one flow of the vector field.  
\end{proof} 

We also prove a sort of converse statement. Let $\tau= \sum_{i=0}^s \tau^i \phi_i$ be a parameter in $H^*(X)$ as before. Recall from Example \ref{exa:J} that $z J_X(\tau)$ is a $\C[\![Q,\tau]\!]$-valued point on $\cL_X$.  We say that a $\C[\![Q,\tau,y]\!]$-valued point $f$ on the Givental cone $\cL_X$ is a \emph{miniversal slice} if it satisfies 
\[
f|_{Q=y=0} = z + \tau+ O(z^{-1}). 
\]
By the definition of $\cL_X$ (see \eqref{eq:point_on_the_cone}), this is equivalent to $f|_{Q=y=0} = z e^{\tau/z}$. 
\begin{lemma} 
\label{lem:miniversal} 
Any miniversal slice of the Givental cone $\cL_X$ can be obtained from $z J_X(\tau)$ by applying a differential operator of the form $\exp(F/z)$ for some $F=F(\tau,z\partial_{\tau},z)$ such that $F\in \sum_{i=0}^s \C[z][\![Q,\tau,y]\!]z \partial_{\tau^i}$ and $F|_{Q=y=0}=0$. 
\end{lemma} 
\begin{proof} 
Let $f\in \cL_X(\C[\![Q,\tau,y]\!])$ be a miniversal slice. By \eqref{eq:cone_as_a_union}, we have $f = z M_X(\sigma) g$ for some $\sigma\in H^*(X)[\![Q,\tau,y]\!]$, $g\in H^*(X)[z][\![Q,\tau,y]\!]$ with $\sigma|_{Q=\tau=y=0} = 0$, $g|_{Q=\tau=y=0}=1$. The miniversality assumption implies that $\sigma|_{Q=y=0} = \tau$ and $g|_{Q=y=0}= 1$. We claim that there exists a vector field $x \in \sum_{i=0}^s \C[\![Q,\tau,y]\!] \partial_{\tau^i}$ such that 
\[
\exp(x) \sigma = \tau, \qquad x|_{Q=y=0} = 0.
\]
This can be shown easily by induction on powers of $Q$ and $y$ (see \cite[Theorem 3.17]{Ilyashenko-Yakovenko} for a related result). Then we have 
\[
\exp(x) f = \exp(x)\left( z M_X(\sigma) g\right) = zM_X(\tau) g'
\] 
with $g' = \exp(x) g$. We have $g'|_{Q=y=0}=1$. Next we claim that there exist $w_0,\dots,w_s \in \C[z][\![Q,\tau,y]\!]$ such that 
\begin{equation} 
\label{eq:exp_w_1}
\exp\left(\sum_{i=0}^s w_i z \nabla_{\tau^i}\right) g' = 1, \qquad w_i|_{Q=y=0}=0.   
\end{equation} 
We construct $w_i$ by induction on powers of $Q$ and $y$. Take an ample class $\omega$ of $X$, and consider the additive valuation $v\colon \C[\![Q,y]\!]\to \R\cup \{\infty\}$ given by $v(Q^d) = \omega \cdot d$ and $v(y_i) = i$. Suppose by induction that there exist $w_0,\dots,w_s\in \C[z][\![Q,\tau,y]\!]$ such that the equation \eqref{eq:exp_w_1} holds modulo terms with valuation $> v_0$ for some $v_0\in v(\C[\![Q,y]\!])$. Let $v_1>v_0$ be the next possible value of the valuation $v$. Write $\exp(\sum_{i=0}^s w_i z \nabla_{\tau^i}) g' = 1-\sum_{i=0}^s \delta_i \phi_i$ for some $\delta_i\in \C[z][\![Q,\tau,y]\!]$ with $v(\delta_i)\ge v_1$. Using the fact that $z \nabla_{\tau^i} 1 = \phi_i$, we have that 
$\exp\left(\sum_{i=0}^s (w_i +\delta_i) z \nabla_{\tau^i} \right) g'$ equals 1 modulo terms with valuation $> v_1$. This completes the induction and the proof of the claim. Setting $G = \sum_{i=0}^s w_i z\partial_{\tau^i}$, we have 
\begin{align*} 
\exp(G) \exp(x) f & = \exp(G) z M_X(\tau) g' =z M_X(\tau) \exp\left( \sum_{i=0}^s w_i z \nabla_{\tau^i}\right) g' \\
& = z M_X(\tau) 1 = zJ_X(\tau)
\end{align*} 
where we used $\partial_{\tau^i} \circ M_X(\tau) = M_X(\tau) \circ \nabla_{\tau^i}$.  
Since $\sum_{i=0}^s \C[z][\![Q,\tau,y]\!] \partial_{\tau^i}$ is closed under commutator, we can rewrite the left-hand side by the Baker-Campbell-Hausdorff formula and obtain $F \in \sum_{i=0}^s \C[z][\![Q,\tau,y]\!] z\partial_{\tau^i}$ such that $\exp(-F/z) f = zJ_X(\tau)$ and $F|_{Q=y=0}=0$, as claimed. 
\end{proof} 
\subsection{Shift operators}
\label{subsec:shift_operator} Shift operators of equivariant parameters have been introduced by Okounkov-Pandharipande \cite{Okounkov-Pandharipande:Hilbert} and studied in  \cite{Maulik-Okounkov, BMO:Springer,Iritani:shift, LJones:shift_symp, GMP:nil-Hecke}. They can be viewed as lifts of the Seidel representation \cite{Seidel:pi1} on quantum cohomology algebras to quantum $D$-modules. In this section we review basic properties of shift operators in the case of the vector bundle $V \to B$. As before, we assume that $V^\vee$ is generated by global sections and thus $V$ is semi-projective by Lemma \ref{lem:semiprojective}. 

In this section, we work with the $\C$-basis $\{\phi_i\lambda^k\}$ and the associated $\C$-linear coordinates $\{\tau^{i,k}\}$ of $H^*_{S^1}(V)$ as in \S\ref{subsec:vector_bundle}. The shift operator shall be defined over the ring $\C[\![\btau]\!]$ of formal power series in $\btau=\{\tau^{i,k}\}_{0\le i\le s, k\ge 0}$. Since $V$ is semi-projective and the fibrewise scalar $S^1$-action is seminegative\footnote{On the other hand, the inverse $S^1$-action is not seminegative: this is the reason why $\bS(\tau)^{-1}$ is not defined without localization.} (in the sense of \cite[Definition 3.3]{Iritani:shift}), the $S^1$-action defines the shift operator \cite[Definition 3.9]{Iritani:shift} 
\[
\bS(\tau) \colon H^*_{S^1}(V) [z][\![Q,\btau]\!] \to H^*_{S^1}(V)[z][\![Q,\btau]\!]. 
\]
This operator satisfies the following properties: 
\begin{align}
\label{eq:bS_properties} 
\begin{split} 
\bS(\tau) (f(\lambda,z) x) & = f(\lambda-z,z) \bS(\tau) x  \\
[\nabla_{\tau^{i,k}},\bS(\tau)] & = [\nabla_{\xi Q\partial_Q},\bS(\tau)] = [\nabla_{z\partial_z}, \bS(\tau)] = 0 \\
M_V(\tau) \circ \bS(\tau) & = \cS \circ M_V(\tau)
\end{split} 
\end{align} 
where $f(\lambda,z) \in \C[\lambda,z][\![Q,\btau]\!]$, $x\in H^*_{S^1}(V)[z][\![Q,\btau]\!]$, $\xi \in H^2(B)$ and $\cS= e_\lambda(V) e^{-z\partial_\lambda}$ is the shift operator \eqref{eq:cS_introd} on the polynomial Givental space $\cH_{V,\rm pol} = H^*_{S^1}(V) \otimes_{\C[\lambda]}\C[\lambda,z^\pm][\![Q]\!]$. Here $e^{-z\partial_\lambda}$ denotes the operator sending $f(\lambda)$ to $f(\lambda -z)$. 
The third property of \eqref{eq:bS_properties} is in \cite[Theorem 3.14]{Iritani:shift} and the second follows from the third together with the differential equations \eqref{eq:diffeq_M} for $M_V$ and the commutativity 
\[
[\partial_{\tau^{i,k}},\cS] = [\xi Q\partial_Q + z^{-1}\xi,\cS] = [z\partial_z + z^{-1} c_1^{S^1}(TV)+ \mu_V,\cS] = 0.  
\]
The shift operator is compatible with the pairing $P_V$ \eqref{eq:P} in the sense that  
\begin{equation} 
\label{eq:shift_pairing} 
e^{-z\partial_\lambda} P_V(\bS(\tau) f,g) = P_V(f, \bS(\tau) g) 
\end{equation} 
with $f,g \in H^*_{S^1}(V)[z][\![Q,\btau]\!]$. This follows from the fact \eqref{eq:pairing_properties} that $M_V(\tau)$ is isometric with respect to $P_V$, the intertwining property \eqref{eq:bS_properties} between $\bS(\tau)$ and $\cS$ and the identity 
\[
e^{-z\partial_\lambda} P_V(\cS f,g) = P_V(f, \cS g) 
\]
that can be verified easily. The operator $\bS(\tau)$ is invertible over the localization $\C(\lambda,z)[\![Q,\btau]\!]$ since the same holds for $\cS$. The precise image of $\bS(\tau)$ is given in the following remark. 

\begin{remark} 
Using \eqref{eq:pairing_properties} and \eqref{eq:bS_properties}, we have 
\begin{align*} 
P_V(\phi_i, \bS(\tau) \phi_j) = P_V(M_V(\tau) \phi_i, \cS M_V(\tau) \phi_j) = \int_B (M_V(\tau,-z) \phi_i) \cup e^{-z\partial_\lambda} (M_V(\tau,z)\phi_j) 
\end{align*} 
where we write $M_V(\tau,z)$ for $M_V(\tau)$. The left-hand side belongs to $\C[\lambda,\lambda^{-1},z][\![Q,\btau]\!]$ whereas the right-hand side belongs to $\C[\lambda,z,z^{-1}][\![Q,\btau]\!]$ by \eqref{eq:M_V_belongs}. Thus $P_V(\phi_i,\bS(\tau) \phi_j)$ belongs to $\C[\lambda,z][\![Q,\btau]\!]$; they define an invertible matrix because $P_V(\phi_i,\bS(\tau)\phi_j)|_{Q=\tau=0} = \int_B \phi_i \cup \phi_j$. This implies that the image of $\bS(\tau)$ is precisely 
\[
\bS(\tau) (H^*_{S^1}(V)[z][\![Q,\btau]\!])=e_\lambda(V) \cdot  H^*_{S^1}(V)[z][\![Q,\btau]\!]. 
\]
The image of $\bS(\tau)$ is a sub $D$-module since $\bS(\tau)$ commutes with the connection. This fact is related to the quantum Serre \cite[Corollary 9]{Coates-Givental}: the above submodule can be identified with the quantum $D$-module of the $(V^\vee,e_\lambda)$-twisted theory \cite[Theorems 2.11, 3.14]{IMM:QSerre}.
\end{remark} 


\section{Proof of the mirror theorem} 
\label{sec:proof_mirrorthm} 
In this section we prove Theorem \ref{thm:mirrorthm_introd} (and a more general version Theorem \ref{thm:mirrorthm} below). Let $\pi\colon \PP(V)\to B$ be the projective bundle associated with a vector bundle $V\to B$ of rank $r\ge 2$ such that $V^\vee$ is globally generated. Let $p= c_1(\cO(1))$ be the first Chern class of the relative $\cO(1)$ over $\PP(V)$. This class defines a splitting of the natural sequence $0\to \Z \to H_2(\PP(V),\Z)\xrightarrow{\pi_*} H_2(B,\Z) \to 0$ as follows: 
\begin{equation} 
\label{eq:H_2_splitting} 
H_2(\PP(V),\Z) = \Z \oplus \Ker p \cong \Z \oplus H_2(B,\Z). 
\end{equation} 
\begin{notation} 
\label{nota:splitting} 
We write $Q$ for the Novikov variable of the base $B$, and $q$ for the Novikov variable of $\PP(V)$ associated with the class of a line in a fibre. Via the above splitting \eqref{eq:H_2_splitting}, we identify the pair $(q,Q)$ with the Novikov variables for $\PP(V)$; more precisely, a curve class $d \in H_2(\PP(V),\Z)$ corresponds to the monomial $q^{p\cdot d} Q^{\pi_* d}$. The assumption that $V^\vee$ is globally generated implies that the class $p$ is nef; consequently the exponent of $q$ in this monomial is nonnegative whenever $d$ is an effective class. 
\end{notation} 

As before, we consider the fibrewise scalar $S^1$-action on $V$ and denote by $\lambda$ the $S^1$-equivariant parameter. Let $\{\phi_i\}_{i=0}^s$ be a basis of $H^*(B)$. Via the canonical isomorphism $H^*_{S^1}(V) \cong H^*(B) \otimes \C[\lambda]$, $\{\phi_i\}_{i=0}^s$ gives a basis of $H^*_{S^1}(V)$ over $\C[\lambda]$. Let $(\tau^0,\dots,\tau^s) \mapsto \tau=\sum_{i=0}^s \tau^i \phi_i$ be the linear coordinates on  $H^*_{S^1}(V)$ associated with $\{\phi_i\}_{i=0}^s$. 

\subsection{Outline} 
\label{subsec:outline}
Since $V^\vee$ is globally generated, we have a surjection $\cO^{\oplus N} \to V^\vee$ for some $N$. Dualizing, we get an exact sequence of vector bundles over $B$: 
\begin{equation} 
\label{eq:VQ-sequence} 
0 \to V \to \cO^{\oplus N} \to \cQ \to 0.  
\end{equation} 
The inclusion $V \hookrightarrow \cO^{\oplus N}$ induces an embedding $\PP(V) \hookrightarrow B \times \PP^{N-1}$. Since $\cQ$ is a quotient of the trivial bundle $\cO^{\oplus N}$, we have a natural section $s$ of $\cQ(1) = \pi_1^*\cQ \otimes \pi_2^*\cO(1)$ over $B\times \PP^{N-1}$ given by the composition $\pi_2^*\cO(-1) \to \cO_{B\times \PP^{N-1}}^{\oplus N} \to \pi_1^*\cQ$, where $\pi_1\colon B\times \PP^{N-1}\to B$ and $\pi_2\colon B\times \PP^{N-1}\to \PP^{N-1}$ are the projections. The section $s$ is transverse to the zero-section and cuts out $\PP(V)$ as the zero-locus. Clearly, $\cQ(1)$ is generated by global sections and therefore $\cQ(1)$ is convex. 

Our starting point is the fact that Theorem \ref{thm:mirrorthm} holds for the trivial bundle $B\times \C^N \to B$: the $J$-function of $B\times \PP^{N-1}$ at parameter $\tau + t p$ (with $\tau \in H^*(B)$) is given by 
\begin{equation} 
\label{eq:J_BtimesP}
J_{B\times \PP^{N-1}}(\tau,t) = \sum_{k=0}^\infty \frac{e^{pt/z} q^k e^{kt}}{\prod_{c=1}^k (p+cz)^N} J_B(\tau).   
\end{equation}
This can be viewed as a special case of the Elezi-Brown mirror theorem \cite{Brown:toric_fibration}. 
As discussed in \S \ref{subsec:QRR}, since $\cQ(1)$ is convex, the Gromov-Witten invariants of $\PP(V)$ can be computed as the non-equivariant limits of the $(\cQ(1),e_\lambda)$-twisted Gromov-Witten invariants of $B\times \PP^{N-1}$. We shall compute a slice of the Givental cone of the $(\cQ(1),e_\lambda)$-twisted theory and take its non-equivariant limit.

\subsection{Proof of Theorem \ref{thm:mirrorthm_introd}} 
\label{subsec:proof} 
We provide a rigorous treatment based on the $\lambda^{-1}$-adic topology, following the method of Remark \ref{rem:modified_Euler_twist}. 
Throughout this section, $\tau$ is assumed to be a non-equivariant class in $H^*(V)=H^*(B)$ (i.e.~independent of $\lambda$). A generalization to equivariant classes will be discussed in Theorem \ref{thm:mirrorthm}. 

We introduce a modified Quantum Riemann-Roch (QRR) operator: 
\begin{align} 
\label{eq:modified_RR}
\begin{split} 
\Delta^\lambda_W &:= e^{\rank(W) (\lambda \log \lambda -\lambda)/z} \Delta_{(W,e_\lambda^{-1})} \\
& = \prod_\delta e^{\frac{(\lambda + \delta) \log (\lambda+\delta) - (\lambda+\delta)}{z}+\frac{1}{2} \log(\lambda + \delta)+ \sum_{n\ge 2} \frac{B_n}{n(n-1)} \left(\frac{z}{\lambda+\delta}\right)^{n-1}} 
\end{split} 
\end{align} 
where $\Delta_{(W,e_\lambda^{-1})}$ is the QRR operator \eqref{eq:symp_operator} of the inverse Euler twist, and $\delta$ ranges over the Chern roots of $W$. This arises from the asymptotic expansion (Stirling's approximation) of the $\Gamma$-function: 
\[
\Delta^\lambda_W
\sim
\prod_{\delta}\sqrt{\frac{z}{2\pi}} z^{\frac{\lambda+\delta}{z}} \Gamma\left(1+\frac{\lambda+\delta}{z}\right). 
\]
The identity $\Gamma(1+x) = x \Gamma(x)$ implies that 
\begin{equation} 
\label{eq:Delta_difference_eq} 
\Delta_W^{\lambda+kz}/\Delta_W^\lambda = \prod_{c=1}^k \prod_{\delta} (\lambda+ \delta+cz).  
\end{equation} 
Ignoring issues of adic convergence (see Remark \ref{rem:conv_modified_RR}), we may regard the modified operator $\Delta_W^\lambda$ as playing the same role as the original  $\Delta_{(W,e_\lambda^{-1})}$, because $e^{c/z}$ preserves the Givental cone for any constant $c$. 

Let $\te_\lambda(\cdot) = \sum_{i\ge 0} \lambda^{-i} c_i(\cdot)$ be the variant of $e_\lambda$ introduced in Remark \ref{rem:modified_Euler_twist} and let $\tDelta_W^\lambda:=\Delta_{(W,\te_\lambda^{-1})}$ denote the QRR operator \eqref{eq:symp_operator} associated with the  $(W,\te_\lambda^{-1})$-twist. We have $\tDelta_W^\lambda=1+O(\lambda^{-1})$. We also set $\tJ_V^\lambda(\tau) := J_V^\lambda(\tau)|_{Q^d \to Q^d \lambda^{c_1(V) \cdot d}}$; this does not contain positive powers of $\lambda$ and gives the $J$-function of the $(V,\te_\lambda^{-1})$-twisted theory by Remark \ref{rem:modified_Euler_twist}. 
We have 
\begin{equation*} 
\Delta_W^\lambda  =e^{(\rank(W)\theta(\lambda) + c_1(W)\log \lambda)/z} \tDelta_W^\lambda 
\quad \text{with $\theta(\lambda) := \lambda \log \lambda - \lambda + \frac{z}{2} \log \lambda$}.  
\end{equation*} 
Introduce a differential operator $f(\lambda)$ in $\tau$ by 
\[
f(\lambda) := \rank(\cQ) \theta(\lambda)-(\log \lambda) z \partial_{c_1(V)}
\]
where $\partial_{c_1(V)}$ denotes the directional derivative along $c_1(V)$ in the $\tau$-space. We consider the differential operator $f(\lambda+z \partial_t)$ in both $\tau$ and $t$, defined via the Taylor expansion in $z\partial_t$, and define its regularization $f_{\rm reg}(\lambda,z\partial_t)$ by subtracting terms involving $\log \lambda$: 
\begin{align*} 
f_{\rm reg}(\lambda, z\partial_t) & := f(\lambda+ z\partial_t) - f(\lambda) - \rank(\cQ) (\log \lambda) z\partial_t \\
& = \rank(\cQ) (\theta(\lambda+ z\partial_t) - \theta(\lambda) - (\log \lambda) z \partial_t) - \log(1+\lambda^{-1}z\partial_t) z\partial_{c_1(V)}.  
\end{align*} 
This belongs to $\C[z \partial_t, z\partial_{c_1(V)}][\![\lambda^{-1}]\!]$ and satisfies $f_{\rm reg}(\lambda,z\partial_t)|_{\lambda^{-1}=0}=0$.

The exact sequence \eqref{eq:VQ-sequence} shows $\tDelta_\cQ^\lambda \tDelta_V^\lambda =1$. Thus $\tDelta_\cQ^\lambda z \tJ_V^\lambda(\tau)$ is a $\C[\![Q,\tau,\lambda^{-1}]\!]$-valued point of $\cL_B$ by the QRR Theorem \ref{thm:QRR}. This is a miniversal slice, i.e.~$\tDelta_\cQ^\lambda z \tJ_V^\lambda(\tau)|_{Q=\lambda^{-1}=0} = z+\tau+O(z^{-1})$. Therefore, we can apply Lemma \ref{lem:miniversal} to this slice, and obtain a differential operator $\tF(\lambda) \in \sum_{i=0}^s \C[\![Q,\tau,\lambda^{-1}]\!] z\partial_{\tau^i}$ such that $\tF(\lambda)|_{Q=\lambda^{-1}=0} = 0$ and 
\begin{equation} 
\label{eq:tilde_key} 
\tDelta_\cQ^\lambda \tJ_V^\lambda(\tau) = e^{\tF(\lambda)/z} J_B(\tau).  
\end{equation} 
Lemma \ref{lem:flow_on_the_cone} and the QRR Theorem \ref{thm:QRR} imply that 
\[
\tI^\lambda(\tau,t) := (\tDelta_{\cQ(1)}^\lambda)^{-1} e^{f_{\rm reg}(\lambda,z\partial_t)/z} e^{\tF(\lambda+z\partial_t)/z} J_{B\times \PP^{N-1}}(\tau,t)
\]
is a $\C[\![Q,\tau,t,\lambda^{-1}]\!]$-valued point of the $(\cQ(1),\te_\lambda)$-twisted Givental cone of $B\times \PP^{N-1}$. 
We will calculate this element $\tI^\lambda$. First observe the following identities: 
\begin{align}
\label{eq:three_identities}  
\begin{split} 
& \tDelta^\lambda_{\cQ(1)} =  e^{(\rank(\cQ)(\theta(\lambda+p)-\theta(\lambda)-p \log \lambda) +c_1(\cQ) \log(1+\frac{p}{\lambda}))/z}\tDelta^{\lambda+p}_\cQ, \\
& \tDelta_\cQ^{\lambda+kz}/\tDelta_\cQ^{\lambda} = 
e^{-\left(\rank(\cQ)(\theta(\lambda+kz) - \theta(\lambda))+c_1(\cQ)\log(1+\frac{kz}{\lambda})\right)/z} \prod_{c=1}^k \prod_\epsilon (\lambda+\epsilon+cz), \\\
& J_V^{\lambda+x}(\tau)\Bigr|_{Q^d \to Q^d \lambda^{c_1(V)\cdot d}} = e^{-c_1(\cQ) \log(1+\frac{x}{\lambda})/z} e^{- \log (1+ \frac{x}{\lambda}) \partial_{c_1(V)}} 
\tJ_V^{\lambda+x}(\tau)  
\end{split} 
\end{align} 
where $\prod_\epsilon$ is the product over the Chern roots of $\cQ$. 
The first identity follows from $\Delta_{\cQ(1)}^\lambda = \Delta_\cQ^{\lambda+p}$, the second from \eqref{eq:Delta_difference_eq}, and the thrid from the Divisor Equation. 
We have 
\begin{align*} 
\tI^\lambda(\tau,t) & = (\tDelta_{\cQ(1)}^\lambda)^{-1} \sum_{k\ge 0} \frac{e^{pt/z} q^k e^{kt }}{\prod_{c=1}^k (p+cz)^N}
e^{f_{\rm reg}(\lambda,p+kz)/z} e^{\tF(\lambda+p+kz)/z} J_{B}(\tau) && \text{by \eqref{eq:J_BtimesP}} \\ 
& =  \sum_{k\ge 0} \frac{e^{pt/z} q^k e^{kt }}{\prod_{c=1}^k (p+cz)^N}
e^{f_{\rm reg}(\lambda,p+kz)/z} (\tDelta_{\cQ(1)}^\lambda)^{-1} \tDelta^{\lambda+p+kz}_\cQ \tJ^{\lambda+p+kz}_V(\tau) && \text{by \eqref{eq:tilde_key}} \\ 
& = \sum_{k\ge 0} e^{pt/z} q^k e^{kt} 
\frac{\prod_{c=1}^k \prod_\epsilon \left(1+\frac{p+\epsilon+cz}{\lambda} \right)}{\prod_{c=1}^k (p+cz)^N} J_V^{\lambda+p+kz}(\tau)\Bigr|_{Q^d\to Q^d \lambda^{c_1(V)\cdot d}} && \text{by \eqref{eq:three_identities}.} 
\end{align*} 
By the change of variables $Q^d \to Q^d \lambda^{c_1(\cQ)\cdot d}=Q^d \lambda^{-c_1(V)\cdot d}$, $q^k \to q^k \lambda^{\rank(\cQ) k}$ for $\tI^\lambda$, we obtain a point $I^\lambda$ of the $(\cQ(1),e_\lambda)$-twisted Givental cone by Remark \ref{rem:modified_Euler_twist}. It is given by 
\[
I^\lambda(\tau,t) = \sum_{k\ge 0} e^{tp/z} q^k e^{kt} 
\frac{\prod_{c=1}^k \prod_\epsilon (\lambda+p+ \epsilon + cz)}{\prod_{c=1}^k (p+cz)^N}  
J_V^{\lambda+p+kz}(\tau). 
\]
The non-equivariant limit $\lim_{\lambda \to 0} i^*I^\lambda$ exists and gives the $I$-function $I_{\PP(V)}$ in Theorem \ref{thm:mirrorthm_introd}. As discussed in \S\ref{subsec:QRR}, $I_{\PP(V)}$ lies on the Givental cone of $\PP(V)$ . The proof of Theorem \ref{thm:mirrorthm_introd} is now complete. 

\begin{remark} 
The differential operator $f(\lambda)$ appearing in the proof above is chosen to satisfy the identity $\Delta_\cQ^\lambda J_V^\lambda(\tau) = e^{f(\lambda)/z}\tDelta_\cQ^\lambda \tJ_V^\lambda(\tau)$, which leads to 
\[
\Delta_\cQ^\lambda J_V^\lambda(\tau) = e^{f(\lambda)/z} e^{\tF(\lambda)/z} J_B(\tau). 
\] 
Conversely, if one chooses a differential operator $F(\lambda) = F(\lambda;\tau,z\partial_\tau,z)$ in $\tau$ such that $\Delta_\cQ^\lambda J_V^\lambda(\tau) = e^{F(\lambda)/z} J_B(\tau)$, then a straightforward calculation using \eqref{eq:Delta_difference_eq} shows that the element $(\Delta_{\cQ(1)}^\lambda)^{-1} e^{F(\lambda+z\partial_t)/z} J_{B\times \PP^{N-1}}(\tau,t)$ yields the point $I^\lambda(\tau,t)$ defined above. This point is expected to lie on the $(\cQ(1),e_\lambda)$-twisted cone of $B\times \PP^{N-1}$, and the preceding construction makes this heuristic rigorous via the $\lambda^{-1}$-adic topology.
\end{remark}

\begin{remark} 
\label{rem:conv_modified_RR} 
We note that the modified QRR operator \eqref{eq:modified_RR} arising from the Stirling approximation does not converge in an adic topology, because the exponent contains\footnote{This also holds for the original operator $\Delta_{(W,e_\lambda^{-1})}$; however this is not an issue because the coefficient of $z^{-1}$ in $\log (\Delta_{(W,e_\lambda^{-1})})$ is nilpotent.} both positive and negative powers of $z$. This converges in the standard topology to power series in $z$ and $z^{-1}$ (infinite in both directions) with coefficients in analytic functions of $\lambda$ near $\lambda =e$, but we are not sure whether this fact would be useful. The properties of $\Delta_W^\lambda$ we used in the proof (e.g.~\eqref{eq:Delta_difference_eq}) can be interpreted in terms of the well-defined logarithm $\log(\Delta^\lambda_W) \in z^{-1}H^*(B)\otimes \C[\lambda, \lambda^{-1},\log \lambda][\![z]\!]$ and the convergence issue does not exist.  
\end{remark}

\subsection{A generalization} 
In this section we observe that the $J$-function $J_V^\lambda$ in Theorem \ref{thm:mirrorthm_introd} can be replaced with any other slices of $\cL_V$ depending polynomially on $\lambda$. 
\begin{theorem}
\label{thm:mirrorthm} 
Let $V\to B$ as in Theorem \ref{thm:mirrorthm_introd}. 
Let $x=(x_1,x_2,\dots)$ be a set of formal parameters and let $z I_V^\lambda$ be a $\C[\lambda][\![Q,x]\!]$-valued point of the Givental cone $\cL_V$. Define the $H^*(\PP(V))$-valued function $I_{\PP(V)}$ by  
\[
I_{\PP(V)} = \sum_{k=0}^\infty 
\frac{e^{pt/z} q^k e^{kt} }{\prod_{c=1}^k \prod_{\substack{\text{\rm $\delta$: Chern roots} \\ \text{\rm of $V$}}} (p+ \delta + cz)} I_V^{p+kz}.    
\]
Then $z I_{\PP(V)}$ lies in the Givental cone $\cL_{\PP(V)}$ of $\PP(V)$.
\end{theorem} 
\begin{proof} 
We can write $I_V^\lambda= M_V(\tau(\lambda)) f(\lambda)$ for some $\tau(\lambda) \in H^*(B)[\lambda][\![Q,x]\!]$ and $f(\lambda)\in H^*(B)[\lambda,z][\![Q,x]\!]$ with $\tau(\lambda)|_{Q=x=0} = 0$ and $f(\lambda)|_{Q=x=0}=1$. We may assume that the basis $\{\phi_i\}_{i=0}^s$ is chosen so that $\phi_0=1$ and write $f(\lambda)=\sum_{i=0}^s f^i(\lambda) \phi_i$. We have 
\[
I_V^\lambda = \sum_{i=0}^s M_V(\tau(\lambda)) f^i(\lambda) \phi_i 
= f^0(\lambda) J_V^\lambda(\tau(\lambda)) + \sum_{i=1}^s f^i(\lambda) (z\partial_{\tau^i}  J_V^\lambda)(\tau(\lambda))
\] 
with $f^0(\lambda)|_{Q=x=0} = 1$, $f^i(\lambda)|_{Q=x=0} = 0$ for $i\ge 1$. 
Writing $\tau(\lambda) = \tau(0) + \delta(\lambda)$, $\delta(\lambda) = \sum_{j=0}^s \delta^j(\lambda) \phi_j$, we have 
\[
I_V^\lambda = f^0(\lambda) \left( e^{\sum_{j=0}^s \delta^j(\lambda) \partial_{\tau^j}} J^\lambda_V\right) (\tau(0)) + \sum_{i=1}^s f^i(\lambda)\left( z\partial_{\tau^i} e^{\sum_{j=0}^s \delta^j(\lambda) \partial_{\tau^j}} J_V^\lambda\right)(\tau(0)) 
\] 
If we write $I_{\PP(V)}(\tau,t)$ for the $I$-function in Theorem \ref{thm:mirrorthm_introd}, the $I$-function in this Theorem \ref{thm:mirrorthm} can be written as 
\begin{align*} 
f^0(z\partial_t) \left( e^{\sum_{j=0}^s \delta^j(z\partial_t) \partial_{\tau^j}}I_{\PP(V)}\right) (\tau(0),t) + \sum_{i=1}^s f^i(z\partial_t) \left( z\partial_{\tau^i}e^{\sum_{j=0}^s \delta^j(z\partial_t) \partial_{\tau^j}} I_{\PP(V)}\right) (\tau(0),t). 
\end{align*} 
We have that $g(\tau,t) :=e^{\sum_{j=0}^s \delta^j(z\partial_t) \partial_{\tau^j}} zI_{\PP(V)}(\tau,t)$ is a $\C[\![Q,\tau,t,x]\!]$-valued point of the Givental cone $\cL_{\PP(V)}$ by Lemma \ref{lem:flow_on_the_cone}. Writing $T$ for the tangent space of $\cL_{\PP(V)}$ at $g$, we have from the overruled property that 
\[
h(\tau,t)=(f^0(z\partial_t)-1) g(\tau,t)+ \sum_{i=1}^s f^i(z\partial_t) z\partial_{\tau^i}g(\tau,t)
\]
gives an element of $zT$ that vanishes at $Q=x=0$; hence $g(\tau,t) + h(\tau,t)$ lies in the cone $\cL_{\PP(V)}$. By specializing $\tau$ to $\tau(0)$ in $g(\tau,t)+h(\tau,t)$, we obtain the $I$-function multiplied by $z$ as a $\C[\![Q,t,x]\!]$-valued point of $\cL_{\PP(V)}$. 
\end{proof}
\begin{example} 
\label{exa:split_case} 
When $V$ is split, the $I$-function $I_{\PP(V)}$ coincides with the Elezi-Brown $I$-function. Suppose that $V = L_1\oplus \cdots \oplus L_r$ with $L_i$ a line bundle over $B$. By the assumption that $V^\vee$ is generated by global sections, each $L_i^{-1}$ is nef. Let 
\[
J_B(\tau) = \sum_{d\in \Eff(B)} J_d(\tau) Q^d 
\]
be the $J$-function of $B$. The hypergeometric modification 
\[
I_V^\lambda = \sum_{d\in \Eff(B)} \left(\prod_{i=1}^r \prod_{c=0}^{-c_1(L_i)\cdot d-1} (\lambda+c_1(L_i) - cz) \right) J_d(\tau) Q^d  
\]
then lies in $\cL_V$ by the quantum Lefschetz theorem \cite{Coates-Givental}. Applying Theorem \ref{thm:mirrorthm} to this $I_V^\lambda$, we recover the Elezi-Brown $I$-function 
\[
I_{\PP(V)} = \sum_{k\ge 0} \sum_{d\in \Eff(B)} e^{pt/z} q^k e^{kt} \left(\prod_{i=1}^r \frac{\prod_{c=-\infty}^{0 \phantom{+c_1(L_i)\cdot d}} (c_1(L_i) + p + cz)}{\prod_{c=-\infty}^{k+ c_1(L_i) \cdot d}(c_1(L_i) + p + cz)} \right) J_d(\tau) Q^d.
\] 
\end{example} 
\begin{example}
\label{exa:cotangent_Pn}
Let $B=\PP^n$ with $n\ge 2$ and consider the cotangent bundle $T^\vee_{\PP^n}$. We compute the $I$-function of $\PP(T^\vee_{\PP^n})$ by using Theorem \ref{thm:mirrorthm}.
The Euler sequence
\begin{equation*}
0 \to T^\vee_{\PP^n}(1) \to \cO^{\oplus(n+1)} \to \cO(1) \to 0,
\end{equation*}
implies that $T^\vee_{\PP^n}(1)$ satisfies the assumption in Theorem \ref{thm:mirrorthm_introd}.
It also implies that $\Delta_{T^\vee_{\PP^n}(1)}^\lambda \Delta_{\cO(1)}^\lambda=\Delta_{\cO^{\oplus(n+1)}}^\lambda$. Since $\Delta_{\cO^{\oplus(n+1)}}^\lambda$ preserves the Givental cone $\cL_{\PP^n}$, we have 
\[
\cL_{\PP^n,(T^\vee_{\PP^n}(1),e_\lambda^{-1})}=\cL_{\PP^n,(\cO(1),e_\lambda)}.
\]
Therefore, by the quantum Lefschetz theorem \cite{Coates-Givental} for the $(\cO(1),e_\lambda)$-twist, 
\[
I_{T^\vee_{\PP^n}(1)}^\lambda=\sum_{d=0}^\infty e^{h\tau/z}Q^d e^{d\tau} \frac{\prod_{c=1}^d (h+\lambda+cz)}{\prod_{c=1}^d (h+cz)^{n+1}}
\] 
lies in the cone of $T^\vee_{\PP^n}(1)$, where $h\in H^2(B,\Z)$ is the ample generator. Applying Theorem \ref{thm:mirrorthm} to $I_{T^\vee_{\PP^n}(1)}^\lambda$, we obtain the $I$-function for $\PP(T^\vee_{\PP^n}) = \PP(T^\vee_{\PP^n}(1))$. 
\[
I_{\PP(T^\vee_{\PP^n})}=\sum_{k=0}^\infty \sum_{d=0}^\infty e^{(pt+h\tau)/z} q^k e^{kt} Q^d e^{d\tau} 
\frac{\prod_{c=1}^{k+d} (p+h+cz)}{\prod_{c=1}^k (p+cz)^{n+1} \prod_{c=1}^d (h+cz)^{n+1}}.
\]
We can obtain this $I$-function in a different way. The space $\PP(T^\vee_{\PP^n})$ is isomorphic to the 2-step flag variety $\Fl(1,n-1;n)$, which can be embedded to $\PP^n \times \PP^n$ as a degree $(1,1)$-hypersurface. The $I$-function $I_{\PP(T^\vee_{\PP^n})}$ can be also obtained as a hypergeometric modification of the $J$-function of $\PP^n\times\PP^n$. 
\end{example}

\section{Fourier duality between quanum $D$-modules} 
\label{sec:Fourier_QDM} 
In this section, we restate the mirror theorem (Theorem \ref{thm:mirrorthm}) as a Fourier duality between the equivariant quantum $D$-module $\QDM_{S^1}(V)$ of $V$ and the quantum $D$-module $\QDM(\PP(V))$ of $\PP(V)$. More precisely, we construct an isomorphism between $\QDM_{S^1}(V)$ and $\htau^* \QDM(\PP(V))$ that intertwines the shift operator $\bS(\tau)$ with the Novikov variable $q$ and the equivariant parameter $\lambda$ with the quantum connection $z (\htau^*\nabla)_{q\partial_q}$. Here $\htau \colon H^*_{S^1}(V) \to H^*(\PP(V))$ is a `mirror map' that equals the Kirwan map modulo higher order terms in $q,Q$. We continue to assume that the dual bundle $V^\vee$ is globally generated. 

\subsection{Quantum $D$-modules}
We fix notation and definitions concerning the quantum $D$-modules of $\PP(V)$ and $V$. 
\subsubsection{Quantum $D$-module of $\PP(V)$} 
\label{subsubsec:QDM_PP(V)}
Let $\htau\in H^*(\PP(V))$ denote the bulk-deformation parameter of the quantum cohomology of $\PP(V)$. Following Notation \ref{nota:splitting}, we write $(q,Q)$ for the Novikov variables of $\PP(V)$, which are dual to the decomposition $H^2(\PP(V),\Z) = \Z p \oplus H^2(B,\Z)$. The \emph{quantum $D$-module} of $\PP(V)$ is the module 
\[
\QDM(\PP(V)) = H^*(\PP(V)) \otimes \C[z,q][\![Q,\htau]\!] 
\] 
equipped with the quantum connection (see \S\ref{subsec:qconn_fundsol})
\[
\nabla_{\htau^i}, \ \nabla_{q\partial_q}, \ \nabla_{\xi Q\partial_Q}, \nabla_{z\partial_z}
\colon \QDM(\PP(V)) \to z^{-1} \QDM(\PP(V)) 
\]
where $\xi \in H^2(B) \subset H^2(\PP(V))$ and the flat $z$-sesquilinear pairing \eqref{eq:P} 
\begin{align*} 
& P_{\PP(V)} \colon \QDM(\PP(V)) \times \QDM(\PP(V)) \to \C[z,q][\![Q,\htau]\!] 
\end{align*} 
induced by the Poincar\'e pairing of $\PP(V)$. We note that $\nabla_{q\partial_q}=q\partial_q + z^{-1} (p\star_{\htau})$. We also note that we defined $\QDM(\PP(V))$ over $\C[z,q][\![Q,\htau]\!]$ instead of over $\C[z][\![q,Q,\htau]\!]$: it is possible because of the homogeneity in Gromov-Witten theory and the fact that both $z$ and $q$ have positive degrees. 

\subsubsection{Quantum $D$-module of $V$} 
\label{subsubsec:QDM_V} 
Let $\tau\in H^*_{S^1}(V)$ denote the bulk-deformation parameter of the equivariant quantum cohomology of $V$. Let $\{\phi_i\}_{i=0}^s$ be a homogeneous basis of $H^*(B)$ as before; here we assume that $\phi_0 = 1$. Then $\{\phi_i \lambda^k\}_{0\le i\le s, k\ge 0}$ gives a $\C$-basis of $H^*_{S^1}(V)$. We write $\tau = \sum_{i=0}^s \sum_{k=0}^{\infty} \tau^{i,k} \phi_i \lambda^k$, where the $\tau^{i,k}$ are $\C$-valued coordinates, and we use $\btau$ to denote the infinite set $\{\tau^{i,k}\}_{0\le i\le s, k\ge 0}$ of coordinates. The \emph{equivariant quantum $D$-module} of $V$ is the module 
\[
\QDM_{S^1}(V) = H^*_{S^1}(V) \otimes_{\C[\lambda]} \C[\lambda,z][\![Q,\btau]\!] = H^*(B) \otimes \C[\lambda,z][\![Q,\btau]\!]
\]
equipped with the quantum connection (see \S\ref{subsec:vector_bundle}) 
\[
\nabla_{\tau^{i,k}}, \ \nabla_{\xi Q\partial_Q}, \ \nabla_{z\partial_z} \colon \QDM_{S^1}(V) \to z^{-1} \QDM_{S^1}(V)
\]
where $\xi \in H^2(B) = H^2(V)$, the shift operator $\bS(\tau) \colon \QDM_{S^1}(V) \to \QDM_{S^1}(V)$  (see \S\ref{subsec:shift_operator}) and the flat $z$-sesquilinear pairing \eqref{eq:P}
\begin{align*} 
& P_V \colon \QDM_{S^1}(V) \times \QDM_{S^1}(V)  \to \C(\lambda,z)[\![Q,\btau]\!] 
\end{align*} 
induced by the Poincar\'e pairing \eqref{eq:pairing_V} of $V$. 

\subsection{Fourier transformation for quantum $D$-modules} 
Recall the Fourier transformation \eqref{eq:Fourier_introd} on the level of the Givental spaces from the introduction: 
\begin{align*} 
\widehat{\phantom{A}} \colon & \cH_{V,\rm pol} \to \cH_{\PP(V),\rm pol} \\ 
& J(\lambda) \longmapsto \hJ(q) = \sum_{k\in \Z} \kappa(\cS^{-k} J) q^k = \sum_{k=0}^\infty \frac{J(p+kz)}{\prod_{c=1}^k e_{p+cz}(V)} q^k, 
\end{align*} 
where $\cH_{V,\rm pol}$, $\cH_{\PP(V), \rm pol}$ are the polynomial Givental spaces given in  \eqref{eq:H_V_pol}, \eqref{eq:H_X_pol}. (We define $\cH_{\PP(V),\rm pol}$ over the completed Novikov ring $\C[\![q,Q]\!]$.)  
\begin{example}[cf.~\eqref{eq:Fourier_J-function}] 
Let $\tau_{\rm ne} = \sum_{i=0}^s \tau^{i,0} \phi_i\in H^*(V)$ denote the non-equivariant class. The Fourier transform of the $J$-function $J_V(\tau_{\rm ne} + \tau^{0,1} \lambda \phi_0)=e^{\lambda \tau^{0,1}/z} J_V(\tau_{\rm ne})$ gives the $I$-function in Theorem \ref{thm:mirrorthm_introd} with $t= \tau^{0,1}$. 
\[
J_V(\tau_{\rm ne}+\tau^{0,1}\lambda \phi_0)\sphat = I_{\PP(V)}(\tau_{\rm ne},\tau^{0,1}). 
\]
More generally, we obtain  
$J_V(\textstyle\sum_{i,k} \tau^{i,k}\phi_i \lambda^k)\sphat = 
e^{\sum_{k\ge 1, (i,k)\neq (0,1)} \tau^{i,k}(zq\partial_q + p)^k \partial_{\tau^i}}
I_{\PP(V)}(\tau_{\rm ne},\tau^{0,1})$
by using \eqref{eq:Fourier_property}.  
\end{example}
We show that the Fourier transformation induces an isomorphism between the quantum $D$-modules of $V$ and $\PP(V)$. 
\begin{theorem}  
\label{thm:Fourier_QDM} 
Let $V\to B$ be as in Theorem $\ref{thm:mirrorthm_introd}$. There exist a mirror map $\htau = \htau(\tau) \in H^*(\PP(V))[q][\![Q,\btau]\!]$ and an isomorphism $\FT\colon \QDM_{S^1}(V) \to \htau^* \QDM(\PP(V))$ of $\C[z][\![Q,\btau]\!]$-modules, such that by defining the pull-back connection $\nabla' = \htau^*\nabla$ via the standard formulae 
\begin{align*} 
\nabla'_{\tau^{i,k}} & = \partial_{\tau^{i,k}} + z^{-1}(\partial_{\tau^{i,k}}\htau)\star_{\htau(\tau)}, \\ 
\nabla'_{q\partial_q} & = q \partial_q + z^{-1} (p\star_{\htau(\tau)}+ (q\partial_q \htau) \star_{\htau(\tau)}), \\ 
\nabla'_{\xi Q\partial_Q} & = \xi Q \partial_Q + z^{-1}(\xi\star_{\htau(\tau)} + (\xi Q\partial_Q \htau)\star_{\htau(\tau)}), \\ 
\nabla'_{z\partial_z} &= z\partial_z - z^{-1} (E_{\PP(V)}\star_{\htau(\tau)})+ \mu_{\PP(V)}, 
\end{align*}
the following properties hold: 
\begin{itemize}
\item[(1)] $\FT$ intertwines $\bS(\tau)$ with $q$; 
\item[(2)] $\FT$ intertwines $\lambda$ with $z \nabla'_{q\partial_q}$;  
\item[(3)] $\FT$ intertwines $\nabla_{\tau^{i,k}}$ with $\nabla'_{\tau^{i,k}}$; 
\item[(4)] $\FT$ intertwines $\nabla_{\xi Q\partial_Q}$ with $\nabla'_{\xi Q\partial_Q}$ for $\xi \in H^2(B)$; 
\item[(5)] $\FT$ intertwines $\nabla_{z\partial_z}+\frac{1}{2}$ with $\nabla'_{z\partial_z}$;  
\item[(6)] $\htau(\tau)$ is homogeneous of degree $2$ and $\FT$ is homogeneous of degree $0$ with respect to the usual grading on $H^*_{S^1}(V)$, $H^*(\PP(V))$ and the following degrees of the variables: $\deg \tau^{i,k} = 2 - \deg \phi_i - 2k$, $\deg Q^d = 2c_1(TV) \cdot d = 2(c_1(TB)+c_1(V))\cdot d$, $\deg q = 2r$ and $\deg z = 2$;  
\item[(7)]  $\htau(\tau)|_{q=Q=0} = \kappa(\tau)$ and $\htau(\tau)|_{Q=0,\tau^{*,r}=\tau^{*,r+1}=\cdots = 0} \equiv \kappa(\tau)$ modulo $(\tau^{*,\le r-1})^2$, where $(\tau^{*,\le r-1})$ denotes the ideal generated by $\tau^{i,k}$ with $0\le i\le s$, $k\le r-1$; 
\item[(8)] $\FT(\phi_i \lambda^k)|_{Q=\tau=0} = \phi_i p^k$ for $0\le k\le r-1$;   
\item[(9)]  $(M_V(\tau) s)\sphat = M_{\PP(V)}(\htau(\tau)) \FT(s)$, i.e.~the following diagram commutes:
\[
\xymatrix{\QDM_{S^1}(V) \ar[r]^{\FT \quad } \ar[d]_{M_V(\tau)} & \htau^*\QDM(\PP(V)) \ar[d]^{M_{\PP(V)}(\htau(\tau))} \\ 
\cH_{V,\rm pol} \ar[r]^{\widehat{\phantom{aa}}} & \cH_{\PP(V),\rm pol}. 
} 
\]
\end{itemize} 
\end{theorem} 
\begin{proof}[Proof of Theorem $\ref{thm:Fourier_QDM}$]
Consider the $J$-function $J_V(\tau)$ with $\tau = \sum_{i,k} \tau^{i,k} \phi_i \lambda^k$. By Theorem \ref{thm:mirrorthm}, the Fourier transform $z \hJ_V(\tau)$ lies in the Givental cone for $\PP(V)$. Hence there exist unique $\htau(\tau)\in H^*(\PP(V))[\![q,Q,\btau]\!]$ and $\Upsilon(\tau)\in H^*(\PP(V))[z][\![q,Q,\btau]\!]$ such that 
\[
(M_V(\tau) 1 )\sphat = \hJ_V(\tau) = M_{\PP(V)}(\htau(\tau)) \Upsilon(\tau). 
\]
By differentiating this in $\tau^{i,k}$, we obtain (using \eqref{eq:diffeq_M})
\begin{equation} 
\label{eq:Fourier_Birkhoff}
(M_V(\tau) \phi_i \lambda^k )\sphat =M_{\PP(V)}(\htau(\tau)) z\nabla'_{\tau^{i,k}} \Upsilon(\tau). 
\end{equation} 
This shows that $M_{\PP(V)}(\htau(\tau))$ arises from the Birkhoff factorization of the matrix with column vectors $\{(M_V(\tau)\phi_i \lambda^k)\sphat : 0\le i\le s, 0\le k\le r-1\}$ because $M_{\PP(V)} = \id + O(z^{-1})$ and $z\nabla'_{\tau^{i,k}} \Upsilon(\tau)$ does not contain negative powers of $z$ (see \cite{Coates-Givental}). Note that 
\begin{equation} 
\label{eq:restr_Q=q=0} 
(M_V(\tau)\phi_i \lambda^k)\sphat\,\big|_{Q=q=0} = e^{\kappa(\tau)/z} \phi_i p^k 
\end{equation} 
and these vectors with $0\le i\le s$, $0\le k\le r-1$ form a basis of $H^*(\PP(V))$. The mirror map $\htau(\tau)$ is then determined by $M_{\PP(V)}(\htau(\tau)) 1 = 1 + z^{-1} \htau(\tau) + O(z^{-2})$. The formula \eqref{eq:restr_Q=q=0} with $i=k=0$ implies that $\htau(\tau)|_{Q=q=0} = \kappa(\tau)$ and $\Upsilon(\tau)|_{Q=q=0} = 1$. Similarly, a direct calculation shows that $(M_V(\tau)1)\sphat\,|_{Q=0,\tau^{*,r}=\tau^{*,r+1}=\cdots =0} \equiv 1 + z^{-1} \kappa(\tau)+O(z^{-2})$ modulo $(\tau^{*,\le r-1})^2$. Part (7) follows from these computations. 

$M_V(\tau)\phi_i\lambda^k$ is homogeneous of degree $\deg (\phi_i\lambda^k)$. It is easy to see that its Fourier transform $(M_V(\tau)\phi_i \lambda^k)\sphat\,$ is also homogeneous of the same degree. This implies the homogeneity of the Birkhoff factors $M_\PP(V)(\htau(\tau))$, $z \nabla'_{\tau^{i,k}}\Upsilon(\tau)$ and hence of $\Upsilon(\tau) = \nabla'_{\tau^{0,0}} \Upsilon(\tau)$ and the mirror map $\htau(\tau)$. The homogeneity implies that $\htau(\tau)$ and $\Upsilon(\tau)$ depend polynomially on $q$ (for each fixed power of $Q$ and $\tau$), i.e.~$\htau(\tau) \in H^*(\PP(V))\otimes \C[q][\![Q,\btau]\!]$, $\Upsilon(\tau) \in H^*(\PP(V))\otimes \C[z,q][\![Q,\btau]\!]$. 

We define a $\C[z][\![Q,\btau]\!]$-linear map $\FT \colon \QDM_{S^1}(V) \to \htau^*\QDM(\PP(V))$ by sending $\phi_i \lambda^k$ to $z\nabla'_{\tau^{i,k}} \Upsilon(\tau)$.  Then $\FT$ is homogeneous of degree zero and satisfies Part (9) by \eqref{eq:Fourier_Birkhoff}. Parts (1)-(5) follow from Part (9) together with the following relations 
\begin{align*}
(\cS f)\sphat & = q \hat{f} \\
(\lambda f)\sphat & = (z q\partial_q +p)\hat{f} \\
(\partial_{\tau^{i,k}} f)\sphat & =\partial_{\tau^{i,k}} \hat{f} \\ 
((\xi Q \partial_Q+z^{-1}\xi) f) \sphat & = (\xi Q\partial_Q+z^{-1}\xi) \hat{f} \\ 
((z\partial_z - z^{-1} c_1^{S^1}(TV) + \mu_V+\tfrac{1}{2}) f )\sphat & = (z\partial_z - z^{-1} c_1(T\PP(V)) + \mu_{\PP(V)}) \hat{f} 
\end{align*} 
and the intertwining properties \eqref{eq:diffeq_M} of $M_V$ and $M_{\PP(V)}$. 

Finally we show that $\FT$ is an isomorphism (along with part (8)). It suffices to show that $\FT(\phi_i \lambda^k)|_{Q=\tau=0}\in H^*(\PP(V))\otimes \C[z,q]$ with $0\le i\le s$, $k\ge 0$ form a basis over $\C[z]$. We have 
\[
(M_V(\tau)\phi_i\lambda^k)\sphat\,\big|_{Q=\tau=0} = \sum_{l=0}^\infty 
\frac{\phi_i (p+lz)^k}{\prod_{c=1}^l e_{p+cz}(V)} q^l = 
\begin{cases} 
\phi_i p^k + O(z^{-1})  & k\le r-1 \\ 
\phi_i p^r + q \phi_i + O(z^{-1}) & k =r 
\end{cases} 
\]
which shows that $\FT(\phi_i \lambda^k) |_{Q=\tau=0}= \phi_i p^k$ for $0\le k\le r-1$ and, since the multiplication by $\lambda$ corresponds to $z\nabla'_{q\partial_q}$ under $\FT$, that 
\[
z\nabla'_{q\partial_q} (\phi_i p^k)\big|_{Q=\tau=0} = 
\begin{cases} 
\phi_i p^{k+1} & 0\le k\le r-2; \\ 
\phi_i p^r + q \phi_i & k=r-1.  
\end{cases}
\]
Hence for $k = rn + m$ with $0\le m\le r-1$, we have
\begin{align*} 
\FT(\phi_i \lambda^k)\big|_{Q=\tau=0} & = (z \nabla'_{q\partial_q})^k \phi_i \big|_{Q=\tau=0}\\
& =  \phi_i p^m  q^n + (\text{polynomial in $q$ of degree $\le (n-1)$}) +O(z). 
\end{align*}  
The homogeneity of $\FT(\phi_i\lambda^k)|_{Q=\tau=0}$ implies that they form a $\C[z]$-basis of $H^*(\PP(V))\otimes \C[z,q]$. 
\end{proof} 

We can also identify the pairing $P_{\PP(V)}$ with a Fourier transform of $P_V$. Define the pairing $\hP_{V}$ on $\QDM_{S^1}(V)$ as follows (see Remark \ref{rem:hP_V} for another definition): 
\begin{equation} 
\label{eq:hP}
\hP_V(x,y) := \sum_{k,l\in \Z} q^{k+l} \Res_{\lambda=0} P_V(\bS(\tau)^{-k} x,\bS(\tau)^{-l} y) d\lambda.  
\end{equation} 
It is clear from the definition that $\hP_V(\bS(\tau)x,y) = \hP_V(x,\bS(\tau)y)=q \hP_V(x,y)$. 
\begin{proposition} 
\label{prop:Fourier_pairing} 
The pairing $\hP_V$ is well-defined and takes values in $\C[z,q][\![Q,\btau]\!]$. We have 
\[
\hP_V(x,y) = P_{\PP(V)}(\FT(x),\FT(y))
\] 
for $x,y \in \QDM_{S^1}(V)$. 
\end{proposition} 
\begin{proof} 
For $x,y\in \QDM_{S^1}(V)$, we set $f =M_V(\tau) x$ and $g = M_V(\tau)y$. We have 
\begin{align*} 
\hP_V(x,y) & = \sum_{k,l\in \Z} q^{k+l} \Res_{\lambda=0} P_V(\cS^{-k} f, \cS^{-l} g) d\lambda && \text{by \eqref{eq:pairing_properties} and \eqref{eq:bS_properties}}\\ 
& = \sum_{k,l\in \Z} q^{k+l} P_{\PP(V)}(\kappa(\cS^{-k} f), \kappa(\cS^{-l}g)) 
&& \text{by Lemma \ref{lem:JK_res} below} \\
& = P_{\PP(V)}(\hat{f}, \hat{g}) \\ 
& = P_{\PP(V)}(M_{\PP(V)}(\htau)\FT(x), M_{\PP(V)}(\htau) \FT(y)) && \text{by Theorem \ref{thm:Fourier_QDM}(9)} \\ 
& = P_{\PP(V)}(\FT(x), \FT(y)) && \text{by \eqref{eq:pairing_properties}}. 
\end{align*} 
Note that the summands are non-zero only when $k,l\ge 0$, and hence the sum is well-defined. It is clear that $P_{\PP(V)}(\FT(x), \FT(y))$ lies in $\C[z,q][\![Q,\btau]\!]$. 
\end{proof} 

\begin{lemma} 
\label{lem:JK_res} 
Let $\pi \colon \PP(V) \to B$ denote the projection. 
For $f = f(\lambda)\in H^*(B)[\![\lambda]\!]$, we have 
\[
\pi_*(\kappa(f) ) = \Res_{\lambda=0} \frac{f(\lambda)}{e_\lambda(V)} d\lambda. 
\]
In particular, by integrating it over $B$, we have 
\[
\int_{\PP(V)} \kappa(f) = \Res_{\lambda=0} \left( \int^{S^1}_V f \right) d\lambda. 
\]
\end{lemma} 
\begin{proof} 
The map $f\mapsto \pi_*(\kappa(f))$ is uniquely characterized by the following properties: 
\begin{itemize}
\item[(1)] it is a homomorphism of $H^*(B)$-modules; 
\item[(2)] it vanishes on $e_\lambda(V) H^*(B)[\![\lambda]\!]$;  
\item[(3)] $\pi_*(\kappa(\lambda^i)) = 0$ for $0\le i\le r-2$ and $\pi_*(\kappa(\lambda^{r-1}))=1$. 
\end{itemize} 
It is easy to check that the right-hand side satisfies the same properties.
\end{proof} 

\begin{remark} 
\label{rem:hP_V}
For $x, y\in \QDM_{S^1}(V)$, we have 
\[
P_V(\bS(\tau)^{-k}x, \bS(\tau)^{-l} y) = e^{-k z \partial_\lambda}P_V(x, \bS(\tau)^{-k-l} y) = e^{-k z\partial_\lambda} P_V(f, \cS^{-k-l} g)
\] 
with $f =M_V(\tau) x$, $g= M_V(\tau) y$ (where we used \eqref{eq:shift_pairing}) and $P_V(f,\cS^{-d}g)=P_V(f,e_{\lambda+z}(V)^{-1} \cdots e_{\lambda+dz}(V)^{-1} e^{dz \partial_\lambda} g)$ has poles only at $\lambda =- jz$ with $j\in [0,d]\cap \Z$ as a function of $\lambda$. Therefore the pairing $\hP_V$ can be rewritten as 
\begin{align*} 
\hP_V(x,y) & = \sum_{d\in \Z} q^d \sum_{k+l=d} \Res_{\lambda=0} e^{-k z\partial_\lambda} P_V(f, \cS^{-d} g) d\lambda \\ 
& = \sum_{d\in \Z} q^d \sum_{0\le k\le d} \Res_{\lambda=-kz} P_V(f, \cS^{-d} g) d\lambda \\ 
& = -\sum_{d\in  \Z} q^d \Res_{\lambda=\infty} P_V(f, \cS^{-d} g) d\lambda \\
& = -\sum_{d\in \Z} q^d \Res_{\lambda=\infty} P_V(x,\bS(\tau)^{-d} y) d\lambda. 
\end{align*} 
\end{remark}

\begin{example} Consider the vector bundle $V=\cO(-1) \oplus \cO^{\oplus(r-1)}$ over $\PP^1$. We set the parameter $\tau$ to be zero and denote by $Q$ the Novikov variable corresponding to the positive generator of $H_2(\PP^1,\Z)$. In this example, we can consider the equivariant quantum $D$-module $\QDM_{S^1}(V)$ over the polynomial ring  $\C[z,\lambda,Q]$; $\QDM_{S^1}(V)$ is generated by $1$ as a $\C[z]\langle Q, z\nabla_{Q\partial_Q},\lambda, \bS \rangle$-module and is defined by the relations 
\begin{align*} 
((z\nabla_{Q\partial_Q})^2 + Q (z \nabla_{Q \partial_Q}-\lambda)) \cdot 1 &=0, \\
(\bS - \lambda^{r-1}(\lambda -z\nabla_{Q\partial_Q})) \cdot 1 &=0.
\end{align*} 
The hyperplane class $h\in H^2(\PP^1)$ is given by $h = z \nabla_{Q\partial_Q} 1$. The quantum connection and the shift operator can be written in the $\C[z,\lambda,Q]$-basis $\{1,h\}$ as 
\[
\nabla = d + \frac{1}{z} \begin{pmatrix} 0 & \lambda Q \\ 1 & -Q \end{pmatrix} \frac{dQ}{Q}, \qquad \bS = \begin{pmatrix} \lambda^r & -\lambda^r Q \\ -\lambda^{r-1} & \lambda^r + Q \lambda^{r-1} \end{pmatrix}\circ e^{-z\partial_\lambda} 
\]
where we omitted the connection in the $z$-direction. Note that the quantum connection does not depend on the rank $r$ but the shift operator does. We identify $\bS$ with the Novikov variable $q$ for $\PP(V)$ associated with the class of a line in the fibre. We can check that $\{1,\lambda,\dots,\lambda^{r-1}, h,h\lambda,\dots, h\lambda^{r-1}\}$ gives a basis of $\QDM_{S^1}(V)$ over  $\C[z,q,Q]=\C[z,\bS,Q]$. The action of $\lambda$ defines the connection $z\nabla_{q\partial_q}$ in the $q$-direction; the Fourier-transformed connection on the $(q,Q)$-space can be written in the above basis as 
\footnotesize  
\[
zd+ \left(
\begin{array}{ccccc|ccccc} 
0 & 0 &  &  & q  & 0 & 0 &  &  & qQ \\  
1 & 0 &      &  & 0 & 0 & 0 &      &  & 0 \\
0 & 1 & &  &   & &   & &  & \\
 &  & \ddots & &  &  &  & & &   \\
0 & & & 1 &  0 & 0 & 0& &  &  0  \\ 
\hline 
0 & 0 &  &  & 0 & 0 & 0 &  &  & q \\ 
 &  &  &  &  & 1 & 0 &  &  & 0 \\  
 &  &  &  & & 0 & 1 & &  & \\
0 & 0 & &  &   0 &   &  & \ddots & &  \\
0 & 0& &   &  1 & 0 & & & 1 &  0 
\end{array} 
\right) \frac{dq}{q} + 
\left( 
\begin{array}{ccccc|ccccc} 
0 & 0 & & & 0& 0 & 0 & & & qQ \\ 
 0  &0   &  &&  &Q & 0 & & &  0\\ 
   & &  & & & 0 & Q & & & \\ 
   & & & & & && \ddots && \\ 
0   & & & & 0& & & & Q & 0  \\ 
\hline 
1 & 0 & & & 0 & -Q & 0 & & & 0 \\ 
0 & 1 & & &    & 0 & -Q & & &  \\ 
  &    & \ddots & &   &  & & \ddots & & \\
  & & & & & & & & -Q & 0\\ 
0  & & & & 1 & 0 & & & 0& 0
\end{array} 
\right) \frac{dQ}{Q}.  
\] 
\normalsize 
The matrices appearing here represent the multiplication by $p$ and $h$ in the quantum cohomology of $\PP(\cO_{\PP^1}(-1)\oplus \cO_{\PP^1}^{\oplus (r-1)})$ with respect to the basis $\{1,p,\dots,p^{r-1},h,hp,\dots,hp^{r-1}\}$. 
\end{example}

\section{Decomposition of the quantum $D$-module of $\PP(V)$} 
\label{sec:decomp} 

In this section we construct a decomposition (Theorem \ref{thm:decomp_QDM_introd}) of the quantum $D$-module of $\PP(V)$ into a direct sum of copies of the quantum $D$-module of $B$. 

\subsection{Statement} We now present a more precise version of Theorem \ref{thm:decomp_QDM_introd}. 
Let $\{\phi_i\}_{i=0}^s$ be a basis of $H^*(B)$ with $\phi_0=1$ as before. Then $\{\phi_i p^k\}_{0\le i\le s, 0\le k\le r-1}$ form a basis of $H^*(\PP(V))$. We denote by $\sigma = \sum_{i=0}^s \sigma^i \phi_i\in H^*(B)$ the bulk parameter for $B$ and by $\htau = \sum_{i=0}^s \sum_{k=0}^{r-1} \htau^{i,k} \phi_i p^k\in H^*(\PP(V))$ the bulk parameter for $\PP(V)$. 
Set 
\begin{equation} 
\label{eq:r'}
r' := \begin{cases} r & \text{when $r-1$ is even;} \\ 
2r & \text{when $r-1$ is odd}. 
\end{cases} 
\end{equation} 
As in Notation \ref{nota:splitting}, we use the pair $(q,Q)$ to denote the Novikov variable of $\PP(V)$. 
Recall from \S\ref{subsubsec:QDM_PP(V)} that the quantum $D$-module $\QDM(\PP(V))$ of $\PP(V)$ is defined over the ring $\C[z,q][\![Q]\!]$. 
Its localized version is defined as  
\[
\QDM(\PP(V))_{\rm loc} := \QDM(\PP(V))\otimes_{\C[z,q][\![Q,\htau]\!]} \C[z](\!(q^{-1/r'})\!)[\![Q,\htau]\!],   
\]
equipped with the induced connection $\nabla$ and the pairing $P_{\PP(V)}$.  

In this section, we introduce an additional copy $Q_B$ of the Novikov variable for $B$ and define the quantum $D$-module of $B$ to be the module 
\[
\QDM(B) = H^*(B)[z][\![Q_B,\sigma]\!] 
\]
equipped with the quantum connection $\nabla_{\sigma^i}, \nabla_{\xi Q_B\partial_{Q_B}}, \nabla_{z\partial_z}$ (with $\xi \in H^2(B)$) and the $z$-sesquilinear pairing $P_B$ (see \S\ref{subsec:qconn_fundsol}). 
We relate the Novikov variable $Q_B$ of $B$ to the variables $(q,Q)$ of $\PP(V)$ via the following formula:  
\begin{equation} 
\label{eq:QB_qQ_relation} 
Q_B^d = q^{-c_1(V)\cdot d/r} Q^d
\end{equation} 
for $d\in H_2(B,\Z)$, which is consistent with the degrees $\deg Q_B^d = 2c_1(TB)\cdot d$, $\deg q=2r$, $\deg Q^d = 2(c_1(TB)+c_1(V))\cdot d$. The formula \eqref{eq:QB_qQ_relation} defines an embedding of the coefficient ring $\C[z][\![Q_B]\!]$ for $\QDM(B)$ into the ring $\C[z][\![q^{-1/r'},q^{-c_1(V)/r}Q] \!]\subset \C[z](\!(q^{-1/r'})\!)[\![Q]\!]$. We introduce the following base changes of $\QDM(B)$:  
\begin{align}
\label{eq:extended_QDM_B} 
\begin{split} 
\QDM(B)_{\rm ext} & := \QDM(B) \otimes_{\C[z][\![Q_B,\sigma]\!]} \C[z][\![q^{-1/r'}, q^{-c_1(V)/r}Q,\sigma]\!],  \\
\QDM(B)_{\rm ext,loc} & := \QDM(B) \otimes_{\C[z][\![Q_B,\sigma]\!]} \C[z](\!(q^{-1/r'})\!)[\![Q,\sigma]\!].  
\end{split} 
\end{align} 
These extended modules are equipped with the quantum connection 
\begin{align*} 
\nabla_{\sigma^i} & = \nabla_{\sigma^i} \otimes 1 + 1 \otimes \partial_{\sigma^i} 
&& = \partial_{\sigma^i} + z^{-1}(\phi_i\star_\sigma) \\ 
\nabla_{q\partial_q} & = \nabla_{-(c_1(V)/r)Q_B\partial_{Q_B}} \otimes 1 + 1 \otimes q\partial_q 
&& = q\partial_q  -z^{-1}( (c_1(V)/r)\star_\sigma) \\ 
\nabla_{\xi Q\partial_Q} & = \nabla_{\xi Q_B\partial_{Q_B}} \otimes 1 + 1\otimes \xi Q\partial_Q 
&& = \xi Q \partial_Q +z^{-1} (\xi\star_\sigma) \\ 
\nabla_{z\partial_z} & = \nabla_{z\partial_z} \otimes 1 + 1\otimes z\partial_z 
&& = z\partial_z - z^{-1} (E_B\star_\sigma) + \mu_B 
\end{align*} 
where $\xi \in H^2(B)$, and the $z$-sesquilinear pairing $P_B$. These structures are induced from those for $\QDM(B)$ via pull-back along the map $(Q,q) \mapsto Q_B$ given by \eqref{eq:QB_qQ_relation}. 

\begin{theorem} 
\label{thm:decomp_QDM} 
There exist formal maps $\varsigma_j \colon H^*(\PP(V)) \to H^*(B)$, $\htau \mapsto \varsigma_j(\htau)$ with  $\varsigma_j(\htau) \in  H^*(B)(\!(q^{-1/r})\!)[\![Q,\htau]\!]$, $j=0,\dots,r-1$ and an isomorphism 
\[
\Phi = \bigoplus_{j=0}^{r-1} \Phi_j  \colon \QDM(\PP(V))_{\rm loc} \cong \bigoplus_{j=0}^{r-1} \varsigma_j^*\QDM(B)_{\rm ext, loc} 
\]
where $\varsigma_j^*\QDM(B)_{\rm ext, loc}$ is the module $H^*(B)\otimes \C[z](\!(q^{-1/r'})\!)[\![Q,\htau]\!]$ equipped with the pull-back quantum connection $\varsigma_j^*\nabla$ 
\begin{align*} 
\varsigma_j^*\nabla_{\htau^{i,k}} & = \partial_{\htau^{i,k}} + z^{-1} (\partial_{\htau^{i,k}}\varsigma_j) \star_{\varsigma_j(\htau)} \\ 
\varsigma_j^*\nabla_{q\partial_q} & = q\partial_q + z^{-1} (-(c_1(V)/r)\star_{\varsigma_j(\htau)} + (q\partial_q \varsigma_j) \star_{\varsigma_j(\htau)}) \\ 
\varsigma_j^*\nabla_{\xi Q\partial_Q} & = Q\partial_Q + z^{-1} (\xi\star_{\varsigma_j(\htau)}+ (\xi Q\partial_Q \varsigma_j) \star_{\varsigma_j(\htau)}) \\ 
\varsigma_j^*\nabla_{z\partial_z} & = z\partial_z - z^{-1} (E_B\star_{\varsigma_j(\htau)}) + \mu_B 
\end{align*} 
and the pairing $P_B$ as in \eqref{eq:P}, such that the following holds with $\lambda_j = e^{2\pi \iu j/r}q^{1/r}$: 
\begin{itemize} 
\item[(1)] $\Phi$ intertwines $\nabla$ with $\bigoplus_{j=0}^{r-1} \varsigma_j^*\nabla$; 
\item[(2)] $\Phi$ intertwines $P_{\PP(V)}$ with $\bigoplus_{j=0}^{r-1} P_B$; 
\item[(3)] $\varsigma_j(\htau)$ is homogeneous of degree $2$ and $\Phi$ is homogeneous of degree $-(r-1)$ with respect to the usual grading on $H^*(\PP(V))$, $H^*(B)$ and the following degree of the variables: $\deg \htau^{i,k} = 2- \deg \phi_i -2k$, $\deg Q^d = 2 (c_1(TB)+c_1(V))\cdot d$, $\deg q = 2r$ and $\deg z=2$; 
\item[(4)] $\varsigma_j(\htau)|_{Q=\htau=0} =r\lambda_j - \frac{2\pi \iu j}{r} c_1(V) + O(q^{-1/r})$; 
\item[(5)] the Jacobian matrix of $\varsigma(\htau)= \bigoplus_{j=0}^{r-1}\varsigma_j(\htau)$ along $Q=\htau=0$ is of the form  $(\partial_{\htau^{i,k}}\varsigma_j)|_{Q=\htau=0} = \lambda_j^k (\phi_i+ O(q^{-1/r}))$ and is invertible over $\C(\!(q^{-1/r})\!)$; and 
\item[(6)] $\Phi_j(\phi_i p^k) |_{Q=\htau=0}= \frac{1}{\sqrt{r}}\lambda_j^{k-(r-1)/2}(\phi_i+O(q^{-1/r}))$. 
\end{itemize} 
\end{theorem} 


\begin{remark} 
\label{rem:splitting_intrinsic_meaning} 
The identification \eqref{eq:QB_qQ_relation} corresponds to a splitting of $\pi_*\colon H_2(\PP(V),\Z)\twoheadrightarrow H_2(B,\Z)$ over $\Q$ defined by the kernel of $c_1(T_{\rm vert}\PP(V)) = r p + \pi^*c_1(V)$. Here $T_{\rm vert}\PP(V) = \Ker(\pi_* \colon T\PP(V) \to \pi^*TB)$ denotes the vertical tangent bundle. This splitting differs from the one in \eqref{eq:H_2_splitting} defined by the kernel of $p$; however, it is intrinsic to the geometry of the projective bundle $\mathbb{P}(V) \to B$ and is independent of the choice of the vector bundle $V$ (up to tensoring with a line bundle). 
\end{remark} 

\begin{remark} 
We considered the quantum $D$-module $\QDM(\PP(V))$ and its decomposition over $\C[z](\!(q^{-1/r'})\!)$. More precisely, these can be defined over the `homogeneous' completion $\C[z](\!(q^{-1/r'})\!)_{\rm hom}$, which consists of finite sums of homogeneous elements in $\C[z](\!(q^{-1/r'})\!)$. Because $z$ and $q$ both have positive degrees, $\C[z](\!(q^{-1/r'})\!)_{\rm hom}$ is contained in the ring  $\C[q^{-1/r'}, q^{1/r'}][\![z]\!]$ of formal power series in $z$. This is relevant to Remark \ref{rem:Orlov}. 
\end{remark} 

\subsection{Idea of the proof} 
We first observe that the modified QRR operator $\Delta_V^\lambda$ \eqref{eq:modified_RR} gives rise to a solution of the difference-differential module $\QDM_{S^1}(V)$. The equation \eqref{eq:Delta_difference_eq} implies that $\Delta_V^\lambda$ solves the following difference equation: 
\begin{equation} 
\label{eq:RR_shift} 
e^{-z\partial_\lambda} \circ (\Delta_V^\lambda)^{-1} = (\Delta_V^\lambda )^{-1} \circ \cS. 
\end{equation} 
The operator $\Delta_V^\lambda$ clearly commutes with $\partial_{\tau^{i,k}}$ and $\xi Q \partial_Q + z^{-1} \xi$ (with $\xi \in H^2(B)$). Therefore the map assigning to a section $s\in \QDM_{S^1}(V)$ the $H^*(B)$-valued function $(\Delta_V^\lambda)^{-1} M_V(\tau) s$ intertwines the actions of $\bS(\tau)$, $\nabla_{\tau^{i,k}}$, $\nabla_{\xi Q\partial_Q}$ with the operators $e^{-z\partial_\lambda}$, $\partial_{\tau^{i,k}}$, $\xi Q\partial_Q + z^{-1} \xi$ respectively (thanks to the intertwining properties \eqref{eq:diffeq_M} of $M_V(\tau)$) and can be regarded as a \emph{solution} for $\QDM_{S^1}(V)$. Therefore its Fourier transform
\[
s\longmapsto \int e^{\lambda \log q/z} (\Delta_V^\lambda )^{-1} M_V(\tau) s \, d\lambda 
\]
should give a solution for $\QDM(\PP(V))$ through the Fourier isomorphism $\htau^*\QDM(\PP(V)) \cong \QDM_{S^1}(V)$ in Theorem \ref{thm:Fourier_QDM}. Despite the convergence issue of $\Delta_V^\lambda$ raised in Remark \ref{rem:conv_modified_RR}, we can formally make sense of this Fourier transformation using the stationary phase approximation (i.e.~the method of steepest descent). The QRR theorem tells us that this lies in the Givental cone of $B$ and induces projections $\QDM(\PP(V)) \to \varsigma_j^* \QDM(B)_{\rm ext, loc}$, $j=0,\dots,r-1$, where the $r$ choices result from the choice of critical points we use for the stationary phase method.

\begin{remark} 
The QRR operator $\Delta_V^\lambda$ induces an isomorphism between the quantum $D$-modules of $V$ and $B$ after localizing $\lambda$. The isomorphism is given by a gauge transformation $R(\tau)\in \lambda^{-r/2} \End(H^*(B))\otimes \C[z](\!(\lambda^{-1})\!)[\![Q,\tau]\!]$ from $\QDM_{S^1}(V)$ to $\sigma^*(\QDM(B)|_{Q_B\to Q})$ satisfying 
\[
M_B(\sigma(\tau)) R(\tau) = (\Delta_V^\lambda)^{-1} M_V(\tau) 
\]
and a coordinate change $\tau \mapsto \sigma = \sigma(\tau)\in r (\lambda -\lambda \log \lambda) - c_1(V) \log \lambda + H^*(B)[\lambda,\lambda^{-1}][\![Q,\tau]\!]$ of the form 
$\sigma(\tau) = \sigma_{\rm pert} + \tau+ O(Q)$ 
with 
\[
\sigma_{\rm pert} = \sum_{\delta:\text{Chern roots of $V$}} \left[(\lambda+ \delta) -  (\lambda+\delta) \log (\lambda+\delta)\right].  
\]
The shift operator acts on $\QDM(B)$ by the trivial one $e^{-z\partial_\lambda}$, but it acts highly non-trivially on the pull-back by the $\lambda$-dependent coordinate change $\tau \mapsto \sigma(\tau)$. From this point of view, the solution $(\Delta_V^\lambda)^{-1} M_V(\tau)$ considered above corresponds to the standard solution $M_B(\sigma(\tau))$ for $\sigma^*(\QDM(B)|_{Q_B\to Q})$.  
\end{remark}

\subsection{Fourier transformation via stationary phase method} 
\label{subsec:stationary_phase} 
In this section we define the stationary phase asymptotics of the Fourier integral 
\begin{equation*} 
\int e^{\lambda \log q/z} (\Delta_V^\lambda)^{-1} J(\lambda) d\lambda 
\end{equation*} 
for $J(\lambda) \in H^*_{S^1}(V)[z,z^{-1}]=H^*(B)[\lambda,z,z^{-1}]$. The integral can be written as 
\[
\int e^{-\varphi(\lambda)/z} 
\lambda^{-c_1(V)/z}\lambda^{-r/2} K(\lambda) d\lambda 
\]
where $\varphi(\lambda) = r (\lambda \log \lambda -\lambda) - \lambda \log q$ and $K(\lambda) = (\tDelta_V^\lambda)^{-1} J(\lambda)$  (see \S\ref{subsec:proof} for $\tDelta_V^\lambda$).  We note that $K(\lambda)$ lies in $H^*(B)[z,z^{-1}](\!(\lambda^{-1})\!)$. We expand the integrand at the critical point $\lambda_0 := q^{1/r}$ of the phase function $\varphi(\lambda)$. We have 
\[
\varphi(\lambda_0 e^{s/\sqrt{\lambda_0}}) =- r \lambda_0 + \frac{r}{2} s^2 +  \varphi_{\ge 3}(s,\lambda_0) \quad \text{with $\varphi_{\ge 3}(s,\lambda_0) = r \sum_{k=3}^\infty 
\frac{k-1}{k!} \frac{s^k}{(\sqrt{\lambda_0})^{k-2}}$} 
\]
and hence the Fourier integral can be rewritten in the form 
\begin{equation} 
\label{eq:Gaussian_form} 
e^{r\lambda_0/z} \lambda_0^{- c_1(V)/z} \lambda_0^{- (r-1)/2}  
\int e^{-r s^2/(2z)} L(s,\lambda_0) ds 
\end{equation} 
with $L(s,\lambda_0) \in H^*(B)[z,z^{-1},s](\!(\lambda_0^{-1/2})\!)$ given by 
\begin{align*} 
L(s,\lambda_0) &=e^{-(\frac{r}{2}-1)u} e^{-u c_1(V)/z}  e^{-\varphi_{\ge 3}(s,\lambda_0)/z} 
K(\lambda_0 e^{s/\sqrt{\lambda_0}}) \qquad \text{with $u= s/\sqrt{\lambda_0}$.} 
\end{align*} 
By expanding $L(s,\lambda_0)$ as a power series in $s$ and performing the Gaussian integral \eqref{eq:Gaussian_form} termwise over $s\in \R$ (assuming $z>0$), we obtain the  \emph{formal stationary phase asymptotics}  
\[
\int e^{\lambda \log q/z} (\Delta_V^\lambda)^{-1} J(\lambda) d\lambda \sim 
\sqrt{2\pi z} \cdot e^{r\lambda_0/z}  \scrF_0(J) 
\]
with $\scrF_0(J)\in \lambda_0^{-(r-1)/2} H^*(B)[z,z^{-1}](\!(\lambda_0^{-1})\!)[\log \lambda_0]$ given by 
\[
\scrF_0(J) = \frac{1}{\sqrt{r}}\lambda_0^{-(r-1)/2} e^{-c_1(V) \log \lambda_0/z}
\left[  e^{z (\partial_s)^2/(2r)} L(s,\lambda_0) \right]_{s=0}.   
\]
Note that half-integer powers of $\lambda_0$ in $L(s,\lambda_0)$ do not contribute. The monodromy transformation on the $q$-plane with respect to the $(2\pi j)$ rotation maps $\lambda_0$ to $\lambda_j := e^{2\pi\iu j/r} q^{1/r}$ and $\scrF_0(J)$ to $\scrF_j(J):=\scrF_0(J)|_{\lambda_0 \to \lambda_j}$. 
We have defined the Fourier transformations $\scrF_j \colon \cH_{V,\rm pol} \to \cH_{B,\C(\!(q^{-1/r'})\!)}$, $j=0,\dots,r-1$ taking values in 
\[
\cH_{B,\C(\!(q^{-1/r'})\!)} :=H^*(B)\otimes \C[z,z^{-1}](\!(q^{-1/r'})\!)[\![Q]\!][\log q] 
\]
where $r'$ is as in \eqref{eq:r'}. 
\begin{proposition}
\label{prop:scrF_properties} 
The map $\scrF_j \colon \cH_{V,\rm pol} \to \cH_{B,\C(\!(q^{-1/r'})\!)}$ satisfies the following properties: 
\begin{itemize} 
\item[(1)] $\scrF_j(\lambda J)  = (\lambda_j + z q\partial_q) \scrF_j(J)$; 
\item[(2)] $\scrF_j(\cS J)  = q \scrF_j(J)$;  
\item[(3)] $\scrF_j((z\partial_z - z^{-1} c_1^{S^1}(TV) + \mu_V+\tfrac{1}{2}) J) = (z\partial_z - z^{-1} (c_1(TB) + r \lambda_j) + \mu_B) \scrF_j(J)$. 
\end{itemize}  
\end{proposition} 
\begin{proof} 
Parts (1), (2) follow by taking the stationary phase asymptotics of the standard properties of the Fourier transformation: 
\begin{align*} 
z q\partial_q \int e^{\lambda \log q/z} (\Delta_V^\lambda)^{-1} J(\lambda) d\lambda 
& = \int e^{\lambda \log q/z} (\Delta_V^\lambda)^{-1} \lambda J(\lambda) d\lambda \\ 
q \int e^{\lambda \log q/z} (\Delta_V^\lambda)^{-1} J(\lambda) d\lambda 
& = \int e^{(\lambda+z) \log q/z} (\Delta_V^\lambda)^{-1} J(\lambda) d\lambda \\
& = \int e^{\lambda \log q/z} (\Delta_V^\lambda)^{-1} \cS J (\lambda) d\lambda  
\end{align*} 
where we shift $\lambda$ as $\lambda \mapsto \lambda-z$ and used \eqref{eq:RR_shift} in the last line. We note that the shift $\lambda \mapsto \lambda -z$ translates into the change of variables 
\[
s \mapsto s + \sqrt{\lambda_j} \log\left(1- z \lambda_j^{-1} e^{-s/\sqrt{\lambda_j}} \right)
\]
in terms of the coordinate $s$ around the critical point (given by $\lambda = \lambda_j e^{s/\sqrt{\lambda_j}}$) and this makes sense in the ring $\C[z,z^{-1},s](\!(\lambda_j^{-1/2})\!)$. To show part (3), observe that the map $J \mapsto e^{c_1(V) \log \lambda_j/z} \scrF_j(J)$ is homogeneous of degree $-(r-1)$ with respect to the usual grading on $H^*_{S^1}(V)$ and $\deg q=2r$, $\deg z =2$. This means 
\[
\scrF_j((z\partial_z + \mu_V + \tfrac{1}{2})J)
=
(z\partial_z + r q\partial_q + \mu_B +  z^{-1} c_1(V) ) \scrF_j(J ).  
\]
Part (3) follows from this together with part (1), the fact that $\scrF_j$ is an $H^*(B)[z,z^{-1}]$-module homomorphism and $c_1^{S^1}(TV) = c_1(TB) + c_1(V) + r\lambda$.  
\end{proof} 

\subsection{Fourier transform of the $J$-function}
In this section we show that the Fourier transform $z\scrF_j(J_V(\tau))$ of the $J$-function of $V$ lies in the Givental cone of $B$. Roughly speaking, the argument goes as follows. The QRR theorem says that $z (\Delta_V^\lambda)^{-1} J_V(\tau)$ ``lies'' in $\cL_B$ modulo the convergence issue. The formal asymptotic expansion of the integral 
\begin{equation} 
\label{eq:Fourier_J}
\int e^{\lambda \log q/z} (\Delta_V^\lambda)^{-1} J_V(\tau) d\lambda
\end{equation} 
is given by applying the differential operator $\sqrt{\frac{2\pi z}{\varphi''(\lambda_j)}}e^{z (\partial_\lambda)^2/(2 \varphi''(\lambda_j))}$ to the integrand multiplied by $e^{\varphi''(\lambda_j)(\lambda-\lambda_j)^2/(2z)}$ and evaluating it at the critical point $\lambda=\lambda_j$. Hence it should lie in $\cL_B$ by a property of the cone (see Lemma \ref{lem:flow_on_the_cone}) up to powers of $z$. A more careful analysis shows the following: 
\begin{proposition} 
\label{prop:Fourier_J}
Write $\tau(\lambda) = \sum_{i,k} \tau^{i,k} \lambda^k \phi_i$ for the parameter in $H^*_{S^1}(V)$. Let   $M_B(\sigma)=M_B(\sigma;Q_B)$ denote the fundamental solution \eqref{eq:fundsol} of the quantum connection of $B$. 
There exist 
\begin{align*} 
\sigma_j&\in - \frac{2\pi\iu j}{r} c_1(V)+ \tau(\lambda_j) + \frakm H^*(B)[\![q^{-1/r}, q^{-c_1(V)/r}Q , q^{\bullet/r} \btau]\!] \\ 
\upsilon_j &\in \frac{1}{\sqrt{r}} \lambda_j^{-(r-1)/2} (1+ \frakm' H^*(B)[z][\![q^{-1/r},  q^{-c_1(V)/r}Q, q^{\bullet/r} \btau]\!])
\end{align*} 
with $\lambda_j = q^{1/r} e^{2\pi \iu j/r}$ such that $\scrF_j(J_V(\tau(\lambda))) =q^{-c_1(V)/(rz)}  M_B(\sigma_j;q^{-c_1(V)/r}Q)\upsilon_j$, where 
\begin{itemize} 
\item $q^{\bullet/r} \btau$ stands for the infinite set $\{q^{k/r} \tau^{i,k}\}$ of variables; 
\item $K[\![ q^{-c_1(V)/r} Q]\!]=\left\{\sum_{d\in \Eff(B)} c_d q^{-c_1(V)\cdot d/r} Q^d: c_d\in K\right\}$ for a module $K$; and
\item $\frakm$ $($resp.~$\frakm'$$)$ is the closed ideal of $\C[\![q^{-1/r}, q^{-c_1(V)/r}Q, q^{\bullet/r} \btau]\!]$ $($resp.~of $\C[z][\![q^{-1/r},q^{-c_1(V)/r}Q, q^{\bullet/r}\btau]\!]$$)$ generated by $q^{-1/r}$ and $q^{-c_1(V)\cdot d/r}Q^d$ with $d\in \Eff(B) \setminus \{0\}$.
\end{itemize}  
Moreover, $\sigma_j$ and $\upsilon_j$ are homogeneous of degree $2$ and $-(r-1)$ respectively with respect to the usual grading on $H^*(B)$ and the following degree of the variables: $\deg \tau^{i,k} = 2 - \deg \phi_i - 2k$, $\deg q=2r$, $\deg Q^d = 2 (c_1(TB)+c_1(V)) \cdot d$ and $\deg z=2$.  
\end{proposition} 
\begin{proof} 
We start with the $J$-function $J_V(\tau)$ with parameter $\tau=\sum_i \tau^i \phi_i$. By Remark \ref{rem:modified_Euler_twist}, $\tJ := J_V(\tau)|_{Q^d \to Q^d \lambda^{c_1(V)\cdot d}}$ is the $J$-function of the $(V,\te_\lambda^{-1})$-twisted theory. Thus $I:=z (\tDelta^\lambda_V)^{-1} \tJ$ is a $\C[\![\lambda^{-1},Q,\tau]\!]$-valued point of $\cL_B$ by the QRR Theorem \ref{thm:QRR}. By substituting $\sum_{k\ge 0} \tau^{i,k} \lambda^k$ for $\tau^i$, we regard it as a $\C[\![\lambda^{-1},Q,\lambda^\bullet \btau]\!]$-valued point on $\cL_B$, where $\lambda^\bullet \btau$ denotes the set $\{\lambda^k \tau^{i,k}\}$ of variables. By the definition of $\scrF_j(J_V(\tau(\lambda)))$ in the previous section, we have 
\begin{multline*} 
z e^{c_1(V) \log \lambda_j/z}\scrF_j(J_V(\tau(\lambda)))\Bigr|_{Q \to \lambda_j^{c_1(V)}Q } \\
= \frac{1}{\sqrt{r}} \lambda_j^{-(r-1)/2}  \left. e^{z (\partial_s)^2/(2r)}\left[ 
e^{-(\frac{r}{2}-1)u} e^{- u c_1(V) /z} e^{-\varphi_{\ge 3}(s,\lambda_j)/z} I 
\biggr |_{\scriptsize \!\!\!\begin{array}{l} Q \to Q e^{-u c_1(V)} \\ 
\lambda \to \lambda_j e^{u} \end{array}} \right] \right|_{s=0} 
\end{multline*} 
where we set $u= s/\sqrt{\lambda_j}$ as before and $Q \to  \lambda^{c_1(V)}Q$ means $Q^d \to \lambda^{c_1(V)\cdot d}Q^d$. The quantity in the big parenthesis $[\cdots]$ belongs to $H^*(B)[z,z^{-1},s][\![\lambda_j^{-1/2},Q,\lambda_j^\bullet \btau]\!]$ and defines a $\C[\![s,\lambda_j^{-1/2},Q,\lambda_j^\bullet \btau]\!]$-valued point of $\cL_B$: the Dilaton and String Equations show that $\cL_B$ is invariant under $e^{-(\frac{r}{2}-1)u} e^{-\varphi_{\ge 3}(s,\lambda_j)/z}$ and the Divisor Equation shows that $\cL_B$ is invariant under $f(Q) \to e^{-u c_1(V)/z} f(Q e^{-u c_1(V)})$. By Lemma \ref{lem:flow_on_the_cone}, the action of $e^{z (\partial_s)^2/(2r)}$ preserves points on the cone\footnote{Strictly speaking, the operator $(z\partial_s)^2/(2r)$ does not vanish at the origin of the formal parameters as required in Lemma \ref{lem:flow_on_the_cone}, but the conclusion holds since the action of $e^{z (\partial_s)^2/(2r)}$ is well-defined.}. Therefore we have $\tsigma_j \in H^*(B)[\![\lambda_j^{-1},Q,\lambda_j^\bullet\btau]\!]$ and $\tupsilon_j \in H^*(B)[z][\![\lambda_j^{-1},Q,\lambda_j^\bullet \btau]\!]$ such that $\tsigma_j\equiv \tau(\lambda_j)$, $\tupsilon_j \equiv 1$ modulo the closed ideal in $\C[\![\lambda_j^{-1},Q,\lambda_j^{\bullet} \btau]\!]$ or in $\C[z][\![\lambda_j^{-1},Q,\lambda_j^\bullet \btau]\!]$ generated by $\lambda_j^{-1}$ and $Q^d$ with $d\in \Eff(B)\setminus \{0\}$ and 
\[
e^{c_1(V) \log \lambda_j/z} \scrF_j(J_V(\tau(\lambda))) \Bigr|_{Q \to \lambda_j^{c_1(V)}Q} = \frac{1}{\sqrt{r}} \lambda_j^{-(r-1)/2} M_B(\tsigma_j) \tupsilon_j. 
\]
As in the proof of Theorem \ref{thm:Fourier_QDM}, $M_B(\tsigma_j)$ and $\tupsilon_j$ can be obtained from the Birkhoff factorization of the matrix formed by the derivatives of the left-hand side with respect to $\tau^{i,0}$. It follows from the homogeneity of $e^{c_1(V)\log \lambda_j/z} \scrF_j(J_V(\tau(\lambda)))$ that $M_B(\tsigma_j)$ and $\tupsilon_j$ are homogeneous of degree zero after the change $Q\to \lambda_j^{-c_1(V)} Q$ of variables. By the Divisor Equation again, we obtain $q^{c_1(V)/(rz)}\scrF_j(J_V(\tau(\lambda))) = M_B(\sigma_j;q^{-c_1(V)/r}Q) \upsilon_j$ with 
\[
\sigma_j = - \frac{2\pi \iu j}{r} c_1(V)  + \tsigma_j\bigr|_{Q\to \lambda_j^{-c_1(V)}Q}, \quad 
\upsilon_j =\frac{1}{\sqrt{r}} \lambda_j^{-(r-1)/2}  \tupsilon_j\bigr|_{Q\to \lambda_j^{-c_1(V)}Q}.
\]
The conclusion follows. 
\end{proof} 


\begin{example} 
Computing the formal stationary phase asymptotics of \eqref{eq:Fourier_J} at $Q=0$ over the Artinian local ring $H^*(B)$, we see that $r\lambda_j+\sigma_j|_{Q=0}-\frac{1}{r} c_1(V) \log q$ is given by the critical value of (cf.~\eqref{eq:modified_RR}) 
\[
f(\lambda) = \tau(\lambda) + \lambda \log q - \sum_{\delta:\text{Chern roots of $V$}} ((\lambda+\delta) \log (\lambda+ \delta) - (\lambda+\delta) ) 
\]
at a unique $H^*(B)$-valued critical point $\hlambda_j\in \lambda_j (1+\lambda_j^{-1}H^*(B)[\![\lambda_j^{-1},\lambda_j^\bullet \btau]\!])$, which is a solution $\lambda$ to the equation $q = e_\lambda(V) e^{-\tau'(\lambda)}$. Here $\hlambda_j$ is an eigenvalue of the quantum product $(p\star)|_{Q=0}$ in the $H^*(B)$-algebra $QH^*(\PP(V))|_{Q=0}$. (It is easy to see that $H^*(B)$ acts by cup product on $QH^*(\PP(V))|_{Q=0}$.) 
\end{example} 

\subsection{Projections induced by the Fourier transformations}
Let $\sigma_j$ be as in Proposition \ref{prop:Fourier_J}. We regard it as a map from $\Spf\C[\![q^{-1/r'},q^{-c_1(V)/r}Q, q^{\bullet/r}\btau]\!]$ to $\Spf \C[\![q^{-1/r'},q^{-c_1(V)/r}Q,\sigma]\!]$, where $r'$ is as in \eqref{eq:r'}, and consider the pull-back of the extended quantum $D$-module $\QDM(B)_{\rm ext}$ \eqref{eq:extended_QDM_B} of $B$ by $\sigma_j$. The pull-back is defined to be the module 
\[
\sigma_j^* \QDM(B)_{\rm ext} := H^*(B)[z][\![q^{-1/r'}, q^{-c_1(V)/r}Q, q^{\bullet/r} \btau]\!] 
\]
equipped with the pull-back $\sigma_j^* \nabla$ of the connection for $\QDM(B)_{\rm ext}$: 
\begin{align*} 
(\sigma_j^*\nabla)_{\tau^{i,k}} &= \partial_{\tau^{i,k}} + z^{-1} (\partial_{\tau^{i,k}} \sigma_j) \star_{\sigma_j} \\ 
(\sigma_j^*\nabla)_{q\partial_q} & = q\partial_q + z^{-1} (-(c_1(V)/r)\star_{\sigma_j} + (q \partial_q \sigma_j) \star_{\sigma_j}) \\ 
(\sigma_j^* \nabla)_{\xi Q\partial_Q} & = \xi Q\partial_Q + z^{-1} (\xi \star_{\sigma_j} 
+ (\xi Q \partial_Q \sigma_j)\star_{\sigma_j}) \qquad (\xi\in H^2(B)) \\ 
(\sigma_j^* \nabla)_{z\partial_z} & = z\partial_z - z^{-1} (E_B\star_{\sigma_j}) + \mu_B 
\end{align*} 
and the pairing $P_B$ as in \eqref{eq:P}. The Divisor Equation shows that these pull-backs are well-defined and define maps 
\[
q^{-k/r} (\sigma_j^*\nabla)_{\tau^{i,k}}, \ (\sigma_j^*\nabla)_{q\partial_q}, \ 
(\sigma_j^* \nabla)_{\xi Q \partial_Q}, \ (\sigma_j^*\nabla)_{z\partial_z} \colon \sigma_j^*\QDM(B)_{\rm ext} \to z^{-1} \sigma_j^* \QDM(B)_{\rm ext}. 
\]
The operator $(\sigma_j^*\nabla)_{\tau^{i,k}}$ has poles along the infinity divisor $(q^{-1/r'} = 0)$, but this is an apparent singularity: $\partial_{\tau^{i,k}}$ viewed as a vector field on $\Spf \C[\![q^{-1/r'},q^{-c_1(V)/r} Q, q^{\bullet/r} \btau]\!]$ has poles of the same order along $(q^{-1/r'}=0)$. 

Similarly to \eqref{eq:extended_QDM_B}, we denote by $\sigma_j^*\QDM(B)_{\rm ext,loc}$ the restriction of $\sigma_j^*\QDM(B)_{\rm ext}$ to the localized base $\Spf\C(\!(q^{-1/r'})\!)[\![Q,\btau]\!]$. 
\begin{align*} 
\sigma_j^*\QDM(B)_{\rm ext, loc} & := \sigma_j^*\QDM(B)_{\rm ext} \otimes_{\C[z][\![q^{-1/r'},q^{-c_1(V)/r}Q, q^{\bullet/r}\btau]\!]} \C[z](\!(q^{-1/r'})\!)[\![Q,\btau]\!] \\
&= H^*(B)[z](\!(q^{-1/r'})\!)[\![Q,\btau]\!]. 
\end{align*} 
\begin{proposition} 
\label{prop:scrF_Pi} 
The Fourier transformation $\scrF_j$ induces a map $\Pi_j \colon \QDM_{S^1}(V) \to \sigma_j^*\QDM(B)_{\rm ext, loc}$ of $\C[z][\![Q,\btau]\!]$-modules such that 
\begin{equation} 
\label{eq:scrF_Pi} 
\scrF_j(M_V(\tau) s) = q^{-c_1(V)/(rz)}M_B(\sigma_j; q^{-c_1(V)/r}Q) \Pi_j(s) 
\end{equation} 
for all $s\in \QDM_{S^1}(V)$. The map $\Pi_j$ satisfies the following: 
\begin{itemize} 
\item[(1)] $\Pi_j$ intertwines $\bS(\tau)$ with $q$; 
\item[(2)] $\Pi_j$ intertwines $\lambda$ with $\lambda_j +z(\sigma_j^*\nabla)_{q\partial_q}$; 
\item[(3)] $\Pi_j$ intertwines $\nabla_{\tau^{i,k}}$ with $(\sigma_j^*\nabla)_{\tau^{i,k}}$; 
\item[(4)] $\Pi_j$ intertwines $\nabla_{\xi Q\partial_Q}$ with $(\sigma_j^*\nabla)_{\xi Q\partial_Q}$ for $\xi \in H^2(B)$; 
\item[(5)] $\Pi_j$ intertwines $\nabla_{z\partial_z}+\frac{1}{2}$ with $(\sigma_j^*\nabla)_{z\partial_z}-r\lambda_j/z$; 
\item[(6)] $\Pi_j$ is homogeneous of degree $-(r-1)$ with respect to the usual grading on $H^*_{S^1}(V)$, $H^*(B)$ and the degree of the variables as in Proposition $\ref{prop:Fourier_J}$; 
\item[(7)] $\Pi_j(\phi_i \lambda^k) \in \frac{1}{\sqrt{r}} \lambda_j^{k-(r-1)/2} 
\left(\phi_i + \frakm H^*(B)[z][\![q^{-1/r},q^{-c_1(V)/r}Q, q^{\bullet/r}\btau ]\!] \right)$
\end{itemize} 
where $\lambda_j = e^{2\pi \iu j/r} q^{1/r}$ and $\frakm$ is as in Proposition $\ref{prop:Fourier_J}$. 
\end{proposition} 
\begin{proof} 
By differentiating the identity $\scrF_j(M_V(\tau)1) =q^{-c_1(V)/(rz)} M_B(\sigma_j;q^{-c_1(V)/r}Q) \upsilon_j$ from Proposition \ref{prop:Fourier_J} by $\tau^{i,k}$ and using the differential equations \eqref{eq:diffeq_M} for $M_V$ and $M_B$, we obtain 
\[
\scrF_j(M_V(\tau) \phi_i \lambda^k) = q^{-c_1(V)/(rz)}M_B(\sigma_j; q^{-c_1(V)/r}Q)  z (\sigma_j^* \nabla)_{\tau^{i,k}} \upsilon_j. 
\] 
Therefore the $\C[z][\![Q,\btau]\!]$-module map $\Pi_j$ sending $\phi_i \lambda^k$ to $z(\sigma_j^* \nabla)_{\tau^{i,k}} \upsilon_j$ satisfies \eqref{eq:scrF_Pi}. Parts (1)-(5) follow immediately from \eqref{eq:scrF_Pi}, Proposition \ref{prop:scrF_properties} and the intertwining properties \eqref{eq:diffeq_M}, \eqref{eq:bS_properties} of $M_V, M_B$. Part (6) follows from the homogeneity of $\sigma_j$, $\upsilon_j$ in Proposition \ref{prop:Fourier_J}. Part (7) follows by a direct computation of 
\[
\Pi_j(\phi_i \lambda^k) = z (\sigma_j^* \nabla)_{\tau^{i,k}} \upsilon_j = z \parfrac{\upsilon_j}{\tau^{i,k}} + \parfrac{\sigma_j}{\tau^{i,k}} \star_{\sigma_j} \upsilon_j 
\]
using Proposition \ref{prop:Fourier_J}. 
\end{proof}

\subsection{End of the proof}
We finish the proof of Theorem \ref{thm:decomp_QDM}. We combine the isomorphism $\FT \colon \QDM_{S^1}(V) \cong \htau^*\QDM(\PP(V))$ in Theorem \ref{thm:Fourier_QDM} with the maps $\Pi_j$ in the above proposition to obtain a map
\[
\bigoplus_{j=0}^{r-1} \Pi_j  \circ \FT^{-1} \colon \htau^*\QDM(\PP(V)) \longrightarrow \bigoplus_{j=0}^{r-1} \sigma_j^* \QDM(B)_{\rm ext, loc}. 
\] 
By Theorem \ref{thm:Fourier_QDM} and Proposition \ref{prop:scrF_Pi}, 
this is a homomorphism of $\C[z,q][\![Q,\btau]\!]$-modules homogeneous of degree $-(r-1)$, and intertwines the connection $\htau^*\nabla$ with 
\[
\bigoplus_{j=0}^{r-1}\left( \sigma_j^*\nabla + \frac{\lambda_j}{z} \frac{dq}{q} - r\lambda_j \frac{dz}{z^2}\right). 
\]
 We note that the pull-back $(\sigma_j+r\lambda_j)^* \nabla$ of the quantum connection of $B$ by $\sigma_j + r\lambda_j$ is well-defined\footnote{Despite the fact that the pull-back $M_B(\sigma_j+r\lambda_j)$ of the fundamental solution is ill-defined in the $q^{-1/r'}$-adic topology.} and that 
\[
\sigma_j^*\nabla + \frac{\lambda_j}{z} \frac{dq}{q} - r\lambda_j \frac{dz}{z^2} 
= (\sigma_j+r\lambda_j)^*\nabla. 
\]
Here we use the fact that the shift $\sigma_j \to \sigma_j + r\lambda_j$ in the $H^0(B)$-direction does not change the quantum product but affects the Euler vector field \eqref{eq:Euler} and the covariant derivative in the $q$-direction (through $q\partial_q(r\lambda_j) = \lambda_j$). Hence we obtain a map 
\begin{equation} 
\label{eq:gluing} 
\bigoplus_{j=0}^{r-1} \Pi_j  \circ \FT^{-1} \colon \htau^*\QDM(\PP(V)) \longrightarrow \bigoplus_{j=0}^{r-1} (\sigma_j+r\lambda_j)^* \QDM(B)_{\rm ext, loc}
\end{equation} 
that is fully compatible with the connection. We now restrict the parameter $\tau$ to lie in $\sfH := \bigoplus_{k=0}^{r-1} H^*(B) \lambda^k \subset H^*_{S^1}(V)$: the mirror map $\tau \mapsto \htau(\tau)$ in Theorem \ref{thm:Fourier_QDM} restricted to $\sfH$ gives an isomorphism between $\sfH$ and $H^*(\PP(V))$ in the formal neighbourhood of $q=Q=0$ since $\htau(\tau)|_{q=Q=0} = \kappa(\tau)$. We compose the map $\sigma_j$ with the inverse $\htau \mapsto \tau=\tau(\htau)$ of this isomorphism and define 
\[
\varsigma_j(\htau) := \sigma_j(\tau(\htau)) + r \lambda_j.  
\]
Because of the homogeneity of $\htau(\tau)$ from Theorem \ref{thm:Fourier_QDM}(6), the inverse mirror map $\tau(\htau)$ belongs to $H^*(B)[q][\![Q,\htau]\!]$. Therefore the composition is well-defined;  $\varsigma_j(\htau)$ belongs to $H^*(B)(\!(q^{-1/r})\!)[\![Q,\htau]\!]$ and is homogeneous of degree $2$. Moreover, the asymptotics in Theorem \ref{thm:Fourier_QDM}(7) and Proposition \ref{prop:Fourier_J} imply that 
\begin{equation} 
\label{eq:varsigma_asymptotics} 
\varsigma_j(\htau)\Bigr|_{Q=0} \in r \lambda_j - \frac{2\pi \iu j}{r} c_1(V) + \alpha_j + \sum_{k=0}^{r-1} \sum_{i=0}^s \htau^{i,k} (\phi_i+\beta_{j,i,k}) \lambda_j^k +  (\htau)^2 H^*(B)
\end{equation} 
for some $\alpha_j, \beta_{j,i,k} \in q^{-1/r} H^*(B)[\![q^{-1/r}]\!]$, where $(\htau)\subset \C(\!(q^{-1/r})\!)[\![\htau]\!]$ denote the ideal generated by $\htau^{i,k}$ with $0\le i\le s$, $0\le k\le r-1$. 
Parts (4) and (5) of Theorem \ref{thm:decomp_QDM} follow from \eqref{eq:varsigma_asymptotics}. 
The above map induces a $\C[z,q][\![Q,\htau]\!]$-module map 
\[
\Phi =\bigoplus_{j=0}^{r-1} \Phi_j \colon \QDM(\PP(V)) \longrightarrow \bigoplus_{j=0}^{r-1} \varsigma_j^* \QDM(B)_{\rm ext, loc}  
\] 
compatible with the connection and homogeneous of degree $-(r-1)$. 
Combining Theorem \ref{thm:Fourier_QDM}(8) and Proposition \ref{prop:scrF_Pi}(7), we have 
\[
\Phi_j(\phi_i p^k)\Bigr|_{Q=\htau=0}\in \frac{1}{\sqrt{r}} \lambda_j^{k-(r-1)/2} \left(\phi_i + q^{-1/r} H^*(B)[z][\![q^{-1/r}]\!]\right)
\]
for $0\le k\le r-1$ and therefore $\Phi$ induces an isomorphism over $\C[z](\!(q^{-1/r'})\!)[\![Q,\htau]\!]$. This is part (6) of Theorem \ref{thm:decomp_QDM}. Inverting $\Phi|_{Q=\htau=0}$, we obtain 
\begin{equation} 
\label{eq:Phi_inverse_asymptotics}
\Phi^{-1}(\phi_i e_k)\Bigr|_{Q=\htau=0} \in \frac{1}{\sqrt{r}} \sum_{j=0}^{r-1} \lambda_k^{(r-1)/2-j}p^j \left(\phi_i +q^{-1/r} H^*(B)[z][\![q^{-1/r}]\!]\right) 
\end{equation} 
where $e_k$ is the vector $(0,\dots,0,1,0,\dots,0)$ (with $1$ in the $k$th position) in $\bigoplus_{j=0}^{r-1} H^*(B)$. 

It remains to show that $\Phi$ preserves the pairing. The asymptotics \eqref{eq:Phi_inverse_asymptotics} shows that 
\begin{align*} 
P_{\PP(V)}& (\Phi^{-1}(\phi_i e_k), \Phi^{-1}(\phi_h e_l))\Bigr|_{Q=\htau=0} \\
& = 
\frac{1}{r} \sum_{j+j'\ge r-1} \lambda_k^{(r-1)/2-j} \lambda_l^{(r-1)/2-j'} \left(\int_{\PP(V)} \phi_i\phi_hp^{j+j'} +O(q^{-1/r})\right) \\
& = P_B(\phi_i e_k, \phi_h e_l) + O(q^{-1/r}). 
\end{align*} 
Writing in the $\C$-bases $\{\phi_i p^k\}$, $\{\phi_ie_k\}$ of $H^*(\PP(V))$ and $\bigoplus_{j=0}^{r-1} H^*(B)$, we regard $P_1 =P_{\PP(V)}$, $P_2 = \bigoplus_{j=0}^{r-1} P_B$ and $\Phi$ as square matrices of size $r \dim H^*(B)$.  Here $P_1$ and $P_2$ are constant matrices and $\Phi$ has entries in $\C[z](\!(q^{-1/r'})\!)[\![Q,\htau]\!]$. The above computation shows that 
\[
P_2^{-1} \Phi^{-T} P_1 \Phi^{-1}|_{Q=\htau=0} = I + O(q^{-1/r})
\]
where $I$ is the identity matrix. Both of the pairings $P_2$ and $\Phi^{-T}P_1 \Phi^{-1}$ are compatible with the connection $\bigoplus_{j=0}^{r-1} \varsigma_j^*\nabla$ (see \eqref{eq:pairing_properties}) and hence $P_2^{-1} \Phi^{-T} P_1 \Phi^{-1}$ is flat with respect to the connection induced on the endomorphism bundle. By \eqref{eq:varsigma_asymptotics} we have 
\[
(\varsigma_j^*\nabla)_{q\partial_q} \Bigr|_{Q=\htau=0} = q\partial_q + \frac{1}{z} \left(\lambda_j - \frac{c_1(V)}{r} + q\partial_q \alpha_j\right) 
\]
where $c_1(V)$ and $q\partial_q \alpha_j \in q^{-1/r} H^*(B)[\![q^{-1/r}]\!]$ acts by cup product. Writing $x= q^{-1/r}$, we find that $v(x) := P_2^{-1}\Phi^{-T}P_1\Phi^{-1} |_{Q=\htau=0}= v_0+v_1 x+ v_2 x^2 +\cdots$ satisfies a differential equation of the form 
\[
(x\partial_x + x^{-1} S + (N_0+N_1 x+ \cdots)) v(x) = 0 
\]
where $S, N_i$ are mutually commuting endomorphisms on $\End(\bigoplus_{j=0}^{r-1} H^*(B))$ such that $S$ is semisimple and $N_0$ is nilpotent. We obtain a sequence of equations  
\[
\tag{$*_n$}
S v_{n+1} + (n +N_0)v_n + N_1v_{n-1}+  N_2 v_{n-2} + \cdots = 0. 
\]
From $S v_0 = N_i v_0= 0$ and ($*_0$) we have $S v_1=0$. Then ($*_1$) implies $S^2 v_2=0$. Since $S$ is semisimple, $Sv_2=0$ and ($*_1$) implies $(1+N_0) v_1 =0$; the nilpotence of $N_0$ implies $v_1=0$. In this way we can show inductively that $v_1=v_2=\cdots =0$. Therefore $P_2^{-1} \Phi^{-T} P_1 \Phi^{-1}|_{Q=\htau=0}=I$; in other words, 
\[
P_1^{-1} \Phi^T P_2 \Phi\Bigr |_{Q=\htau=0} = I. 
\]
We now use the fact that $P_1^{-1} \Phi^T P_2 \Phi$ is flat with respect to the connection on the endomorphism bundle of $\QDM(\PP(V))$. It satisfies a differential equation in $(Q,\htau)$ with at most logarithmic singularities at $Q=0$ and nilpotent residues. Expanding $P_1^{-1} \Phi^T P_2 \Phi$ in power series in $Q$ and $\htau$ and using a similar argument as before, we can easily show that $P_1^{-1}\Phi^T P_2 \Phi = I$ identically. Thus $\Phi$ preserves the pairing. 
\begin{remark} 
The map \eqref{eq:gluing} glues the connections $\htau^*\QDM(\PP(V))$ and $\bigoplus_{j=0}^{r-1} (\sigma_j+r\lambda_j)^*\QDM(B)_{\rm ext}$; $\htau^*\QDM(\PP(V))$ is defined around $q=Q= \btau=0$ and $(\sigma_j+r\lambda_j)^*\QDM(B)_{\rm ext}$ is defined around $q^{-1/r'} = q^{-c_1(V)/r}Q = q^{\bullet/r} \btau = 0$. They glue to a connection over a global ``K\"ahler moduli space'' obtained by adding the infinity divisor $(q=\infty)$ to $\Spf( \C[q][\![Q,\btau]\!])$. See Figure \ref{fig:decomposition}.  
\end{remark} 

\subsection{Semisimplicity} 
\label{subsec:semisimplicity} 
Since the decomposition $\Phi$ in Theorem \ref{thm:decomp_QDM} intertwines  $\nabla_{\htau^{i,k}}$ with $\bigoplus_{j=0}^{r-1} (\varsigma_j^*\nabla)_{\htau^{i,k}}$, it follows that 
\[
(\phi_i p^k)\star_{\htau} = \Phi^{-1} \left( \textstyle\bigoplus_{j=0}^{r-1} 
(\partial_{\htau^{i,k}} \varsigma_j) \star_{\varsigma_j(\htau)}\right) \Phi 
+ z \Phi^{-1} \partial_{\htau^{i,k}} \Phi.   
\]
Setting $z=0$, this shows that the map sending $\phi_i p^k$ to $\bigoplus_{j=0}^{r-1}\partial_{\htau^{i,k}} \varsigma_j$ gives a ring isomorphism $(H^*(\PP(V)),\star_{\htau}) \cong \bigoplus_{j=0}^{r-1} (H^*(B),\star_{\varsigma_j(\htau)})$ over $\C(\!(q^{-1/r})\!)[\![Q,\htau]\!]$. This proves Corollary \ref{cor:decomp_qcoh}. 

Corollary \ref{cor:decomp_qcoh} can be applied to show the semisimplicity of big quantum cohomology of certain partial flag varieties. Let $Fl(k_1,\dots,k_l;n)$ denote the partial flag variety consisting of flags $V_1\subset V_2 \subset \cdots \subset V_l \subset \C^n$ with $\dim V_i = k_i$. Fix a non-degenerate skew-symmetric 2-form $\omega$ on $\C^{2n}$. Let $IFl(k_1,\dots,k_l;2n)$ denote the isotropic partial flag variety consisting of flags $V_1\subset V_2\subset \cdots \subset V_l \subset \C^{2n}$ such that $\dim V_i = k_i$ and that $V_l$ is isotropic with respect to $\omega$. Here we necessarily have $k_l\le n$. The varieties $G(k,n) = Fl(k;n)$ or $IG(k,2n) = IFl(k;2n)$ are also known as ordinary or isotropic Grassmannians. We have the following result. 
\begin{proposition} 
\label{prop:semisimplicity} 
The partial flag varieties $Fl(k_1,\dots,k_l;n)$ and $IFl(k_1,\dots,k_l;2n)$ have generically semisimple big quantum cohomology. 
\end{proposition} 
\begin{proof} 
In view of Corollary \ref{cor:decomp_qcoh}, it suffices to show that these spaces are connected to a point by a sequence of projective bundles. For example, $IG(3,6)= IFl(3;6)$ fits into the sequence:
\[
IFl(3;6) \xleftarrow{\PP^2} IFl(2,3;6) \xrightarrow{\PP^1} IFl(2;6) \xleftarrow{\PP^1} IFl(1,2;6) \xrightarrow{\PP^3} IFl(1;6) = \PP^5 \to \pt.  
\] 
The general case is similar. 
\end{proof} 

\begin{remark} 
It is expected that rational homogeneous varieties $G/P$ have semisimple big quantum cohomology. We refer the reader to \cite{Perrin-Smirnov} and references therein for recent developments on this question. 
\end{remark} 

\begin{figure}[t]
\centering 
\begin{tikzpicture}[x=1.2pt, y=1.2pt] 

\fill (-120,0) circle [radius=1.5];   
\fill (-180,0) circle [radius =1.5]; 
\fill (-135,26) circle [radius =1.5]; 
\fill (-165,26) circle [radius =1.5]; 
\fill (-135,-26) circle [radius =1.5]; 
\fill (-165,-26) circle [radius =1.5]; 

\draw (-150,-50) node {$Q=\htau=0$, $q>0$}; 

\fill[opacity=0.1] (30,0) circle [radius=10];   
\fill[opacity=0.1] (-30,0) circle [radius =10]; 
\fill[opacity=0.1] (15,26) circle [radius =10]; 
\fill[opacity=0.1] (-15,26) circle [radius =10]; 
\fill[opacity=0.1] (15,-26) circle [radius =10]; 
\fill[opacity=0.1] (-15,-26) circle [radius =10]; 

\filldraw (34,0) circle [radius =1]; 
\filldraw (28,3.4) circle [radius=1]; 
\filldraw (28,-3.4) circle [radius=1]; 
\filldraw (18.9,26.7) circle [radius=1]; 
\filldraw (12.4,29) circle [radius=1]; 
\filldraw (13.6,22.3) circle[radius=1]; 
\filldraw (-11.9,28.57) circle [radius=1]; 
\filldraw (-18.75, 27.36) circle [radius=1]; 
\filldraw (-14.3,22.06) circle [radius=1]; 
\filldraw (-28,3.46) circle [radius=1]; 
\filldraw (-34,0) circle [radius =1]; 
\filldraw (-28,-3.46) circle [radius =1]; 

\filldraw (18.9,-26.7) circle [radius=1]; 
\filldraw (12.4,-29) circle [radius=1]; 
\filldraw (13.6,-22.3) circle[radius=1]; 
\filldraw (-11.9,-28.57) circle [radius=1]; 
\filldraw (-18.75,-27.36) circle [radius=1]; 
\filldraw (-14.3,-22.06) circle [radius=1]; 

\draw (0,-50) node {small deformation of $(Q,\htau)$}; 

\end{tikzpicture} 
\caption{Decomposition of eigenvalues of the Euler multiplication of $\PP(V)$ into $r$ groups ($r=6$ in this picture).}  
\label{fig:Euler_eigenvalues} 
\end{figure}
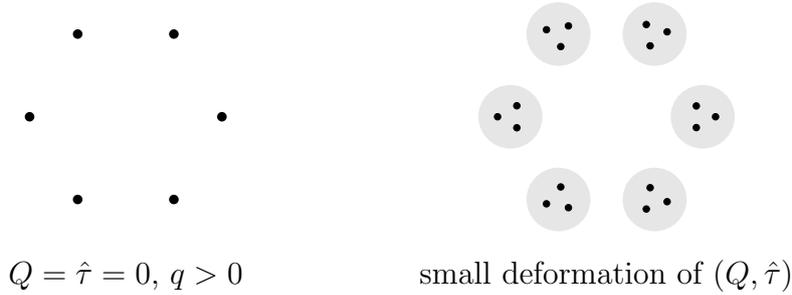 

\begin{remark} 
We can describe a decomposition of the eigenvalues of the quantum multiplication by the Euler vector field of $\PP(V)$ along $Q=\htau=0$. Because $z \nabla_{z\partial_z} = E_{\PP(V)}\star_{\htau} + O(z)$ for the quantum connection $\nabla$ of $\PP(V)$, Theorem \ref{thm:decomp_QDM} implies that the eigenvalues of $E_{\PP(V)}\star_{\htau}$ are the union of the eigenvalues of $E_B\star_{\varsigma_j(\htau)}$, $j=0,\dots,r-1$. 
When $Q=\htau=0$, the quantum product $\star_{\varsigma_j(\htau)}$ for $B$ equals the cup product and the eigenvalues of $E_B\star_{\varsigma_j(\htau)}$ are all equal to the $H^0$ component $r \lambda_j=r e^{2\pi\iu j/r} q^{1/r}$ of $E_B$ along $Q=\htau=0$ (see Theorem \ref{thm:decomp_QDM}(4)). See Figure \ref{fig:Euler_eigenvalues} for a typical decomposition of eigenvalues. Similar pictures have appeared in \cite[Figure 16]{Iritani:discrepant}, \cite[Figure 1]{Koto:convergence}. 
\end{remark} 

\subsection{Reconstruction method of Hinault-Yu-Zhang-Zhang}
\label{subsec:initial_condition}
The algorithm of Hinault-Yu-Zhang-Zhang \cite{HYZZ:framing} (see also \cite{Kontsevich:Miami2020,Kontsevich:Simons2021}) reconstructs the decomposition isomorphism $\Phi$ and the coordinate change $\htau \mapsto \varsigma_j(\htau)$ in Theorem \ref{thm:decomp_QDM} in terms of the genus-zero Gromov-Witten invariants of $B$, along with the initial condition along the locus $Q=\htau=0$. In this section, we describe the initial condition required for their reconstruction in terms of the Chern classes of the vector bundle $V\to B$. This provides, in particular, an algorithm to reconstruct the genus-zero Gromov-Witten invariants of $\PP(V)$ from those of $B$ and the Chern classes of $V$ (see Remark \ref{rem:calculation_GW_of_PV}). 

We describe the initial conditions for $\Phi$ and $\varsigma_j$ along $Q=\htau=0$. Recall that the isomorphism $\Phi=\bigoplus_{j=0}^{r-1} \Phi_j$ arises from the restriction of the isomorphism $\bigoplus_{j=0}^{r-1} \Pi_j \circ \FT^{-1}$ in \eqref{eq:gluing} to the subspace $\sfH = \bigoplus_{k=0}^{r-1} H^*(B)\lambda^k$. Since $\htau(\tau)|_{Q=\tau=0} = 0$ by Theorem \ref{thm:Fourier_QDM}(7), we have $\Phi_j|_{Q=\htau=0}=\Pi_j \circ \FT^{-1}|_{Q=\tau=0}$. We know by Theorem \ref{thm:Fourier_QDM}(8) that $\FT^{-1}(\phi_i p^k)|_{Q=\tau=0} = \phi_i\lambda^k$. Therefore we have 
\[
\Phi_j(\phi_i p^k)|_{Q=\htau=0} = \Pi_j(\phi_i \lambda^k)|_{Q=\tau=0} = e^{-\sigma_j|_{Q=\tau=0}/z} q^{c_1(V)/(rz)} \scrF_j(\phi_i \lambda^k)
\]
by Proposition \ref{prop:scrF_Pi}. Consider the case $\phi_i p^k =1$. Since $q^{c_1(V)/(rz)}\scrF_j(1)$ is of the form $\frac{1}{\sqrt{r}} \lambda_j^{-(r-1)/2} e^{2\pi\iu j c_1(V)/(rz)}(1+O(q^{-1/r}))$, its logarithm with respect to the cup product structure on $H^*(B)$ is well-defined. Because $\Pi_j(1)$ contains no negative powers of $z$, the initial value $\varsigma_j^\circ := \varsigma_j(\htau)|_{Q=\htau=0} = r \lambda_j + \sigma_j|_{Q=\tau=0}$ for $\varsigma_j$ can be calculated as 
\begin{equation} 
\label{eq:initial_cond_varsigma} 
\varsigma_j^\circ = r\lambda_j + [z^{-1}] \log\left( q^{c_1(V)/(rz)} \scrF_j(1) \right) 
\end{equation} 
where $[z^{-1}](\cdots)$ denotes the coefficient of $z^{-1}$. The initial value $\Phi_j^\circ := \Phi_j|_{Q=\htau=0}$ for $\Phi_j$ is then given by 
\begin{equation} 
\label{eq:initial_cond_Phi} 
\Phi_j^\circ(\phi_i p^k) = e^{-(\varsigma_j^\circ-r\lambda_j)/z} q^{c_1(V)/(rz)} \scrF_j(\phi_i \lambda^k) 
\end{equation} 
for $0\le i\le s$ and $0\le k \le r-1$. Note that $\scrF_j(\phi_i \lambda_k)$ can be calculated explicitly in terms of the Chern classes of $V$ via the procedure in \S\ref{subsec:stationary_phase}. 

We briefly describe the reconstruction of $\Phi$ and $\varsigma_j$ from their initial conditions \eqref{eq:initial_cond_varsigma}, \eqref{eq:initial_cond_Phi} by means of Birkhoff factorization. We leave the detailed verification to the reader, as the argument is completely parallel to the case of blowups treated in \cite[\S 5.8]{Iritani:monoidal}. We set: 
\begin{equation} 
\label{eq:s_j}
s_j(\htau) = \varsigma_j(\htau) - \varsigma_j^\circ.  
\end{equation} 
The formal change of variables $\htau \mapsto (s_0(\htau),\dots,s_{r-1}(\htau))$ between $H^*(\PP(V))$ and $H^*(B)^{\oplus r}$ is invertible over $\C(\!(q^{-1/r})\!)[\![Q]\!]$. Thus, we may treat $s_j = s_j(\htau)$, $j=0,\dots,r-1$ as independent variables instead of $\htau$. 
Introduce the block-diagonal endomorphism\footnote{This serves as a fundamental solution for $\QDM(B)_{\rm ext, loc}^{\boxplus r}$ with respect to the variables $(Q,s_0,\dots,s_{r-1})$ (but not for $q$ and $z$).} $M\in \End(H^*(B)^{\oplus r})[z^{-1}](\!(q^{-1/r'})\!)[\![Q,s_0,\dots,s_{r-1}]\!]$ as
\[
M = \bigoplus_{j=0}^{r-1} e^{-\varsigma_j^\circ/z} M_B(\varsigma_j^\circ+s_j; Q q^{-c_1(V)/r}).  
\]
This satisfies $M|_{Q=s_0=\dots=s_{r-1}=0} = \id$ and $M=\id +O(z^{-1})$. 
Then the composition $(\Phi^\circ)^{-1} \circ M$ admits a unique Birkhoff facorization of the form 
\[
(\Phi^\circ)^{-1} \circ M = M' \circ \Phi^{-1} 
\]
where $M'=\id +O(z^{-1})$ is a power series in $z^{-1}$ and $\Phi^{-1}$ is a power series in $z$ such that 
\begin{align*} 
M' &\in \End(H^*(\PP(V)))\otimes \C[z^{-1}](\!(q^{-1/r'})\!)[\![Q,s_0,\dots,s_{r-1}]\!],  \\ 
\Phi^{-1}& \in \Hom(H^*(B)^{\oplus r}, H^*(\PP(V)))\otimes \C[z](\!(q^{-1/r'})\!)[\![Q,s_0,\dots,s_{r-1}]\!],  
\end{align*} 
with $M'|_{Q=s_0=\cdots=s_{r-1}=0}=\id$ and $\Phi^{-1}|_{Q=s_0=\cdots =s_{r-1}=0} = (\Phi^\circ)^{-1}$. 
The inverse of $\Phi^{-1}$ gives the desired isomorphism $\Phi$ written in the coordinates $(s_0,\dots,s_{r-1})$. Furthermore, $M'$ equals  $(M_{\PP(V)}(\htau)|_{Q=\htau=0})^{-1} M_{\PP(V)}(\htau)$ in the coordinates $(s_0,\dots,s_{r-1})$. Finally, the change of coordinates $\htau=\htau(s_0,\dots,s_{r-1})$ inverse to \eqref{eq:s_j} is determined by the asymptotics $M' 1 = 1 + \htau(s_0,\dots,s_{r-1})/z + O(z^{-2})$.

\bibliographystyle{amsplain}
\bibliography{qcoh_projective_bundle}
\end{document}